\documentclass{tran-l}

\usepackage{amsmath}   
\usepackage{amssymb}
\usepackage{hyperref}
\newcommand{\note}[1]{}
\renewcommand\marginpar[1]{}

\newtheorem{theorem}{Theorem}[section]
\newtheorem{lemma}[theorem]{Lemma}

\theoremstyle{definition}

\theoremstyle{remark}

\numberwithin{equation}{section}

\newcommand{\rh}{\mathcal{B}}
\newcommand{\crease}{\Lambda}
\newcommand{\comment}[1]{}

\newcommand{\sring}[1]{\Sigma_{#1}}
\newcommand{\funsol}{\Xi}
\newcommand{\average}{-\!\!\!\!\!\!\int}

\newcommand{\ball}[2]{B_{#2}(#1)}
\newcommand{\sball}[2]{\Delta_{#2}(#1)}

\newcommand{\locdom}[2]{\Omega_{#2}(#1)}
\newcommand{\nontan}[1]{#1^*}

\newcommand{\ntar}[1]{\Gamma(#1)}
\newcommand{\neugre}{{\mathcal N}}
\newcommand{\reals}{\mathbf{R}}

\newcommand{\dball}[2]{\Psi _{ #2}(#1)}
\newcommand{\dist}{\mathop{\rm dist}\nolimits}
\renewcommand{\div}{{\mathop{\rm div}\nolimits}}
\newcommand{\luspace}{{\mathcal B}}
\newcommand{\osc}{\mathop{\rm osc}}

\newcommand{\sobolev}[2]{W^{#2, #1}}
\providecommand{\tangrad}{\nabla _t}

\begin{document}

\title[The mixed problem with general decompositions of the boundary]
{The mixed problem in Lipschitz domains with
general decompositions of the boundary}


\author{J.L.~Taylor}
\address{Department of Mathematics, Murray State University, Murray, Kentucky}
\curraddr{}
\email{jtaylor52@murraystate.edu}

\author{K.A.~Ott}
\address{Department of Mathematics, University of Kentucky, Lexington, Kentucky}
\email{katharine.ott@uky.edu}
\thanks{Research supported, in part, by the National
Science Foundation.}

\author{R.M.~Brown}
\address{Department of Mathematics, University of Kentucky, Lexington, Kentucky}
\email{russell.brown@uky.edu}
\thanks{Research supported, in part, by a grant from the Simons Foundation.}


\date{\today}


\begin{abstract}
This paper continues the study of the mixed
problem for the Laplacian. We consider a bounded Lipschitz
domain $\Omega\subset \reals^n$, $n\geq2$, with boundary
that is decomposed as $\partial\Omega=D\cup N$,
$D$ and $N$ disjoint. We let $\Lambda$ denote the
boundary of $D$ (relative to $\partial\Omega$) and impose
conditions on the dimension and shape of $\Lambda$ and the sets $N$
and $D$.
Under these geometric criteria, we show that there exists
$p_0>1$ depending on the domain $\Omega$ such that for $p$ in the interval $(1,p_0)$,
the mixed problem with Neumann data in the space $L^p(N)$ and
Dirichlet data in the Sobolev space $\sobolev p 1(D) $ has a unique
solution with the  non-tangential maximal function of the gradient of
the solution
 in $L^p(\partial\Omega)$. We also obtain results for $p=1$
when the Dirichlet and Neumann data comes from Hardy spaces, and
a result when the boundary data comes from weighted
Sobolev spaces.
\end{abstract}

\maketitle

\section{Introduction}

In this paper we consider the mixed problem, or Zaremba's
problem, for the Laplacian.
Let $\Omega$ be a bounded
Lipschitz domain in $\reals ^n$ and write $\partial\Omega=D\cup N$,
where
$D$ is an open set of the boundary and
$N=\partial\Omega\setminus D$.
We define the \textit{$L^p$-mixed problem} as
the following boundary value problem
\begin{equation} \label{MP}
\left\{
\begin{array}{ll}
\Delta u = 0, \qquad & \mbox{in }  \Omega\\
u = f_D, \qquad &  \mbox{on }  D\\[2pt]
\frac { \partial u }{ \partial \nu} = f_N,\qquad
& \mbox{on } N \\[2pt]
(\nabla u)^* \in L^p (\partial\Omega).
\end{array}
\right.
\end{equation}
Here, $(\nabla u)^*$ stands for the non-tangential maximal
function of $\nabla u$. The normal
derivative $\partial u/ \partial \nu$ is defined as $\nabla u \cdot \nu$,where
$\nu$ is the outward unit normal vector defined a.e.
on $\partial\Omega$.
Throughout the paper, all boundary values of
$u$ and $\partial u/\partial \nu$ are defined
as non-tangential limits.
See Section 2 for precise definitions.

The study of the mixed problem is a natural
continuation of the program of study of boundary
value problems in Lipschitz domains which began
over thirty years ago. Dahlberg \cite{BD:1977}
treated the Dirichlet problem for the Laplacian in
Lipschitz domains, while Jerison and Kenig
\cite{JK:1982c}
treated the Neumann problem with boundary data in
$L^2$ and the regularity
problem with Dirichlet data having one derivative in $L^2$.
Verchota \cite{GV:1984} studied the regularity
problem with Dirichlet data with one derivative in $L^p$, and Dahlberg
and Kenig
studied the Neumann problem with $L^p$ data
\cite{DK:1987}.   The mixed boundary value
problem in Lipschitz domains appears as an open
problem in
Kenig's CBMS lecture notes
\cite[Problem ~3.2.15]{CK:1994}.
There is a large literature on boundary value problems in polyhedral
domains and we do not attempt to summarize this work here.
See the work of B{\u{a}}cu{\c{t}}{\u{a}}{\em
  et.~al.~}\cite{MR2735986} 
for recent results for the mixed problem in
polyhedral domains and additional references.

Under mild restrictions on the boundary data we can use energy
estimates to show that there exists a solution of the mixed problem
with $\nabla u$ in $L^2$ of the domain. Our goal in this paper is to
obtain more regularity of the solution and, in particular, to show
that $ \nabla u$ lies in $L^p(\partial \Omega)$.
Brown \cite{RB:1994b} showed that the solution
satisfies $\nabla u \in L^{2}(\partial\Omega)$
when the data $f_N$ is in $L^{2}(N)$ and $f_D$
is in the Sobolev space $\sobolev 2 1 (D)$ for a certain
class of Lipschitz domains. Roughly speaking, his
results hold when the Dirichlet and Neumann portions
of the boundary meet at an angle strictly less than
$\pi$. In this same class of domains, Sykes and
Brown \cite{SB:2001} obtain $L^p$ results for $1<p<2$ and
I.~Mitrea and M.~Mitrea \cite{MM:2007} establish well-posedness in an essentially
optimal range of function spaces.
Lanzani, Capogna and Brown \cite{LCB:2008} establish
$L^p$ results in two dimensional graph domains when
the data comes from weighted $L^2$-spaces and the
Lipschitz constant is less than one. The aforementioned
results rely on a variant of the
Rellich identity. The Rellich identity cannot be used
in the same way in general Lipschitz domains because it
produces estimates in $L^2$, and even in smooth domains
simple examples show that we cannot expect to have
solutions with gradient in $L^2(\partial\Omega)$.

Ott and Brown \cite{OB:2009} establish conditions on
$\Omega$, $N$, and $D$ which ensure uniqueness of solutions
of the $L^p$-mixed problem
and they also establish conditions on $\Omega$, $N$, $D$ and
$f_N$ and $f_D$ which ensure that solutions to the $L^p$-mixed
problem exist. All of this work is done under an additional
geometric assumption on the boundary of $D$. More specifically,
the authors address solvability of the mixed problem for
the Laplacian in bounded Lipschitz domains under the
assumption that the boundary between $D$ and $N$
(relative to $\partial\Omega$) is locally given by a Lipschitz
graph. Under these conditions on $\Omega$, $N$, and $D$, they
prove that there exists $p_0>1$ depending on the Lipschitz
constant of the domain and on the dimension $n$, such that
for $p$ in the interval $1<p<p_0$, the $L^p$-mixed
problem with Neumann data in $L^p (N)$ and Dirichlet
data in the Sobolev space $\sobolev p 1 (D)$ has a solution
and this solution is unique in the class of functions
satisfying $(\nabla u)^* \in L^{p}(\partial\Omega)$. In
the case $p=1$, they prove results for the mixed problem
with data from Hardy spaces. The novelty of this
paper is to address
existence and uniqueness of solutions of the $L^p$-mixed
problem under more general conditions on  the decomposition of the
boundary into sets $N$ and $D$.  
Our proof relies on a technique of Shen \cite{ZS:2007} to use reverse
H\"older inequalities to establish existence of solutions to the $L^p$-mixed problem.
This technique
allows for an immediate extension to the mixed boundary value problem
with data from weighted spaces. We carry out a study of the mixed
problem in weighted spaces in Section 7. As one step of this argument,
we need to consider the regularity problem with boundary data
in weighted $L^p$-spaces.

The boundary between $D$ and $N$ is an important feature
of the domain in the study of the mixed problem. Assume 
that $D$ is a relatively open subset of
$\partial\Omega$ and let $\Lambda$ denote the boundary
of $D$ (relative to $\partial\Omega$). Before stating our
assumptions on $\Omega$, $N$, and $D$, we
introduce the following notation. We will use
$\delta(y)=\mbox{dist}(y,\Lambda)$ to denote the
distance from a point $y$ to $\Lambda$. Let
$\ball x r = \{y: |y-x|<r\}$ denote the standard ball
in $\reals ^n$ and let $\Psi_r(x)=B_r(x)\cap \Omega$.
For $ x\in \bar \Omega$, 
let $\sball x r =
B_r(x)\cap \partial\Omega$ denote a {\it surface ball}. 
We note that the term surface ball is not
ideal since the ``center'' $x$ may not lie on the boundary. In
addition, we will need to be careful in places because  $\sball x r$
may not be a connected set. 
See Section 2 for other relevant definitions.

Our assumptions on $\Omega$  and $D$ are stated
here. We will obtain results only when the parameter $ \epsilon$ in
(\ref{SurfProp}) is small. See
Section 2 for a definition of the constant $r_0$.
\begin{equation}
\label{Lip}
\Omega \subset \reals^n, n\geq 2, \, \mbox{is a bounded Lipschitz
  domain of constant $M$}.
\end{equation}
The set 
 $\crease$ is an \textit{Ahlfors $(n-2+\epsilon)$--regular set}: There exists
$M>0$ such that
\begin{equation}\label{SurfProp}
 \mathcal{H}^{n-2+\epsilon}(\sball x r \cap \crease)
\leq M r^{n-2+\epsilon}, \quad \mbox{for all}\,\,
x\in \crease,\ 0<r< r_0,
\end{equation}
with $\epsilon \geq 0$. The notation $\mathcal{H}^{n-2+\epsilon}(E)$ denotes
the $(n-2+\epsilon)$--dimensional Hausdorff measure of a set $E$.
Our third main assumption is that 
the set $D$ satisfies the corkscrew
condition relative to $ \partial \Omega$. There exists $M>0$ such that
\begin{equation} 
\label{NTA}
\parbox[t]{4in}
{
for all $ x\in \crease,\,\, 0<r< r_0$, 
there exists $  \tilde{x}\in D $ such that 
$ |x-\tilde{x}| < r$ and  $ \delta(\tilde{x})>
 M^{-1}r$ .}
\end{equation}
\comment{Alternative conditions
\begin{eqnarray}
\mbox{for all}\,\, x\in \crease,\,\, 0<r< r_0,\,\,
\mbox{there exists} \,\, \tilde{x}\in D\,\,\mbox{such that}
\,\, \nonumber\\
 |x-\tilde{x}|\leq r\,\,\mbox{and}\,\, \sball {\tilde x}  {M^ { -1}r}
 \subset D,\label{NTA}\\
\mbox{for all}\,\,x\in \crease,\,\, 0<r<r_0,\,\,
\mbox{there exists}\,\, \tilde{x}\in N\,\,\mbox{such that}\,\,
\nonumber \\
|x-\tilde{x}|\leq r\,\,\mbox{and}\,\,
\sball {\tilde x}  {M^ { -1}r}
 \subset N,\label{NTA2}
\end{eqnarray}
}

Several previously studied cases of the mixed problem fall under the
conditions of assumptions (\ref{Lip}), (\ref{SurfProp}), and
(\ref{NTA}). Venouziou and Verchota \cite{VV:2008} establish a
solution to the $L^p$-mixed problem (\ref{MP}) in polyhedral domains
in $\reals^3$. In one particular case, they are able to solve the
mixed boundary value problem in the pyramid in $\reals^3$, when
Dirichlet and Neumann data are assigned to alternating faces. This example
is not covered by the earlier work of Ott and Brown \cite{OB:2009} because
at the apex of the pyramid, the boundary between $D$ and $N$ is 
not locally given by a Lipschitz graph. The pyramid example is covered by
the results in this paper. Another
example that is covered by this work, but not the earlier work of Ott and
Brown \cite{OB:2009}, is the case where the boundary of $D$
is a Koch snowflake of dimension slightly greater than $n-2$.

We now state the main theorem of the paper. The definitions
are given in Section \ref{prelim}.  Since we do not
assume that the Dirichlet set $D \subset \partial\Omega$ is an extension
domain for Sobolev spaces, note that we must assume that the Sobolev space $W^ {
  1,p}( D)$ is defined by restricting elements in $W^ {1,p}(\partial\Omega)$
to $ D$.  See Section \ref{prelim} for a discussion of the constants in the
estimates of this theorem.

\begin{theorem}\label{main}
Let $\Omega$ and $D$ satisfy conditions
(\ref{Lip}), (\ref{SurfProp}), and (\ref{NTA}). There exists
an exponent $q_0 > 2$, which depends on  $M$ and $n$,  such that if $\Lambda$
satisfies (\ref{SurfProp}), with
$0 \leq  \epsilon < (q_0 -2 )/(q_0  - 1)$,  such that the following 
statements are true.

a) If $ p \geq 1$, the  $L^p$-mixed problem has at most one solution.

b) If $f_N$ lies in $H^1(N)$ and $f_D$ lies in $H^{1,1}(D)$,
the $L^1$-mixed problem has a  solution $u$ which satisfies
the estimate
$$
\| (\nabla u )^*\| _ { L^ 1 (\partial \Omega )} \leq C ( \| f_N\|_ { H
  ^ 1( N) } + \|f_D\| _ { H ^ { 1,1 } (D)}) .
$$

c) If $p_0=q_0((1-\epsilon)/(2-\epsilon)) >1$, then for $p$ in the
interval $(1,p_0)$ the following holds: If $f_N$ lies in $L^{p}(N)$ and $f_D$ lies
in $\sobolev p 1 (D)$, then the $L^p$-mixed problem has a solution $u$
which satisfies
%
$$
\| (\nabla u ) ^ * \| _ { L^ p ( \partial \Omega ) }
\leq C ( \| f_N \| _{L^ p(N) } + \|f_D \| _ { \sobolev p1  (D)} ) .
$$
\end{theorem}

The proof of the main theorem will proceed as follows.
We begin by recognizing that it suffices to prove Theorem \ref{main}
in the case where the Dirichlet data is zero. This is because
non-tangential
maximal function estimates for the gradient of the solution
to the Dirichlet problem are known. When the Dirichlet data comes
from a Sobolev space, these estimates were treated for
$p=2$ by Jerison and Kenig \cite{JK:1982c} and  for  $ 1< p <  2$
by Verchota \cite{GV:1982,GV:1984}.  The case of the Dirichlet
problem with data from a Hardy space was treated by
Dahlberg and Kenig \cite{DK:1987} and by D.~Mitrea in two dimensions
\cite[Theorem 3.6]{MR1883390}.

The first main result presented in the paper is existence of solutions of the mixed
problem when the Neumann data is an atom for a Hardy space. The proof
of this result is contained in Section 4. The key ingredient of the
proof is an estimate of the Green function for the mixed
problem, which is proved in Section 3. In Section 5 we prove uniqueness of solutions
to the $L^p$-mixed problem, $p\geq 1$.
Section 6 contains the proof of the $L^p$ result and Section 7 contains
the proof of the weighted result.

{\em Acknowledgment. }The authors thank the referee for their
helpful remarks.  Part of this work was carried out while Russell
Brown was visiting the Mathematical
Sciences Research Institute in Berkeley, California, whose hospitality
is gratefully acknowledged.

\section{Preliminaries}\label{prelim}

Throughout the paper, we will work under the assumption (\ref{Lip})
that $ \Omega$ is a bounded Lipschitz domain. A bounded, connected open set $\Omega$ is called a
Lipschitz domain with Lipschitz constant $M$
if locally $\Omega$ is a domain which lies above the graph of a Lipschitz
function. More precisely,
for $M>0$, $x\in \partial\Omega$,
and $r>0$, define a \textit{coordinate cylinder} $Z_r(x)$ to
be $Z_r(x) = \{ y: |y'-x'|< r, |y_n-x_n|< (1+M)r \}$. Use
coordinates $(x',x_n)\in \reals^{n-1}\times \reals$ and assume
that this coordinate system is a translation and rotation of the
standard coordinates. Then $\Omega$ is a \textit{Lipschitz
domain} if for each $x\in \partial\Omega$ there exists a coordinate
cylinder and a Lipschitz function $\phi:\reals^{n-1}\rightarrow \reals$
with Lipschitz constant $M$ such that
\begin{eqnarray*}
\Omega\cap Z_r(x) & = & \{ (y', y_n): y_n> \phi(y')\}\cap Z_r(x),\\
\partial\Omega \cap Z_r(x) & = & \{(y', y_n): y_n = \phi(y') \}\cap Z_r(x).
\end{eqnarray*}
Fix a covering of the boundary by coordinate cylinders $\{Z_{r_i}
(x_i)\}_{i=1}^{L}$ such that each $Z_{100r_i\sqrt{1+M^2}}(x_i)$ is also
a coordinate cylinder. Let $r_0 = \min\{r_i: i=1,\ldots ,L\}$.

For a Lipschitz domain $\Omega$ we define a \textit{decomposition
of the boundary for the mixed problem}, $\partial\Omega = D\cup N$ as follows.
Assume that $D$ is a relatively open subset of $\partial\Omega$,
$N=\partial\Omega \setminus D$, and let $\Lambda$ be the boundary of
$D$ (relative to $\partial\Omega$). The assumptions on the decomposition
of the boundary for the mixed problem are given in (\ref{SurfProp}) and (\ref{NTA}).
Recall that $\delta(y) = \dist(y,\Lambda)$ denotes the distance from a point $y$ to the boundary
of $D$.

Many of our estimates will be of a local, scale invariant nature and
will hold for $ r $ less than a multiple of $ r_0$, and with a
constant that depends only on the constant $M$ 
in assumptions (\ref{Lip})-(\ref{NTA}), $\epsilon $ in
(\ref{SurfProp}), the dimension $n$, and any $L^p$-indices that appear
in the estimate.  We say that an estimate depends on the global
character of the domain if it depends on the above and also on the 
collection of coordinate cylinders which cover $\partial
\Omega$ and the constant in the coercivity condition (\ref{coerce}).
The notation $A \approx B$ will mean that $c^{-1}B\leq A \leq cB$
for some constant $c$ depending only on $M$ and $n$.

We now prove several consequences of the conditions (\ref{SurfProp}) and
(\ref{NTA}) that we will appeal to later in the paper. 

\begin{lemma}\label{newball}
Let $\Omega$ satisfy (\ref{Lip}) and let $r$ be such that
$0<r<r_0$. If $x\in \partial\Omega$ satisfies $\delta(x) \geq  r\sqrt{1+M^2}$,
then $\sball x r \subset N$ or $\sball x r \subset D$.
\end{lemma}

\begin{proof} Let $x \in \partial\Omega$ and fix $y\in \sball x r$.
Since $0<r<r_0$, we may find a coordinate cylinder $Z$ which contains
$\sball x r$. Let $\phi$ be the function whose graph gives $\partial\Omega$
near $Z$. Since $y\in \sball x r$, we have that $|y'-x'|< r$. Now define a
function $x':[0,1]\rightarrow \reals^{n-1}$,
$x'(t) = (1-t)x'+ty'$. Then $\gamma(t) = (x'(t),\phi(x'(t)))$ gives a
path contained in $\partial\Omega$ that connects $x$ and $y$ and has
length less than  $r\sqrt{1+M^2}$. Since $\delta(x)\geq r\sqrt{1+M^2}$, and $\delta$ is
Lipschitz with constant one, we have that $\delta(\gamma(t))>0$
for $0\leq t \leq 1$. Since $\gamma(t)$ does not pass through $\Lambda$, we have
that both $x$ and $y$ must lie in either $D$ or $N$. The point $y$ was arbitrarily
chosen in $\sball x r$, therefore $\sball x r\subset D$ or $\sball x r \subset N$.
\end{proof}

The following lemma is adapted from a result
found in Lehrb\"{a}ck \cite[pp.~254-255]{MR2442898}.
Here and throughout the paper, we use $\sigma$ to denote
surface measure.

\begin{lemma}\label{measure}
Let $\Omega$  and $D$ satisfy
(\ref{Lip}) and (\ref{SurfProp}) and let
$r$ satisfy $0<r<r_0$. Then for each $x\in \Lambda$ and 
$0<t<2r$, 
\begin{equation}
\sigma\left( \sball x r \cap \{y: \delta(y)< t\} \right)
\leq C t^{1-\epsilon}  r^{n-2+\epsilon}.
\end{equation}

\end{lemma}

\begin{proof}
Fix $x\in\crease$,  $0<r<r_0$ and $ t $ with $ 0 < t < 2r$.
By a standard covering lemma, there exists
a finite, disjoint collection of surface balls
$\{\sball {y_i} t\}$ with $y_i\in \sball x r \cap \crease$,
$i=1,2,\ldots, m$, such that
 $(\sball x r \cap \crease)
\subset \bigcup_{i=1}^m \sball {y_i} {3t}$. Using   that the collection of
balls $ \{ \Delta _ t ( y _i ) \} _ { i =1} ^ m $ is disjoint and  the
$(n-2+\epsilon)$--regularity of $\crease$ given in
(\ref{SurfProp}), we have
\begin{eqnarray*}
t^{n-2+\epsilon} m & \leq & C\sum_{i=1}^{m} \mathcal{H}^{n-2+\epsilon}\left(\sball {y_i} t \cap \crease \right) \\
& \leq & C \mathcal{H}^{n-2+\epsilon}\left(\sball x {3r} \cap \crease\right) \\
& \leq & C r^{n-2+\epsilon}.
\end{eqnarray*}
This calculation yields the estimate
$m \leq  m_t = C(r/t)^{n-2+\epsilon}$ for $t<2r$,
where $C$ depends on $M$ and the dimension $n$.

Now,
\begin{eqnarray*}
\sigma\left(\sball x r \cap \{y: \delta(y)< t \} \right)
& \leq & C\sum_{i=1}^{m_t} \sigma\left(\sball {y_i} {4t} \right) \\
& \leq & Ct^{n-1}m_t  \\
& \leq & Ct^{n-1}\left(\frac{r}{t}\right)^{n-2+\epsilon}
= Ct^{1-\epsilon}r^{n-2+\epsilon},
\end{eqnarray*}
which proves the Lemma.
\end{proof}

\begin{lemma}
\label{CorkEverywhere}
Let $\Omega$ and $D$ satisfy (\ref{Lip}) and (\ref{NTA}).  There exists
a constant $c$ such that if  
$x\in D$ and $ 0 < r < r_0$, then there exists
$x_D \in  D $ such that $|x-x_D | \leq r$
and $\sball {x_D} {c r} \subset D$.  Furthermore,
\begin{equation}
\sigma (  \sball x r \cap D  )  \geq  c  r^ { n-1}. \label{Dbig}
\end{equation}
\end{lemma}

\begin{proof}
Let $x\in D$ and $0<r<r_0$. We will break up the proof into two cases.
First suppose that $r/2<\delta(x)$, and let $x_D = x$. Let $ Z_r(x_D)$
be the coordinate cylinder centered at $x_D$ with radius $r$. Using
Lemma \ref{newball} we see that $ \delta (x_D ) > r/2$ implies that
$\sball x {c_1 r} \subset D$ for $ c_1=  1/(2\sqrt{
  1+M^2})$.

Now consider the case where $\delta(x)\leq r/2$.  According to
(\ref{NTA}) and Lemma \ref{newball}, given $\hat{x}\in\Lambda$, there
exists $\bar x \in D $ 
such that $|\hat{x}-\bar x |\leq r/2$ and $\sball { \bar x }{
  c_2 r } \subset D$ with $c_2= 1/ (2M\sqrt{1+M^2})$.  Recall that
$\delta(x)\leq r/2$, and choose
$\hat{x}$ on $\Lambda$ so that $|x-\hat{x}|=\delta(x)$. By the remark
above, there exists $\bar{x} $ with $|\hat{x}-\bar{x} |\leq r/2$ and
$\sball {\bar x }{ c_1r } \subset D$. Let $x _D =\bar{x}$. Then
\begin{eqnarray*}
|x-x_D | & \leq & |x-\hat{x}| + |\hat{x}-   x _D | \\
& \leq &  r.
\end{eqnarray*}
Thus, we obtain $\sball x {cr} \subset D$ if $ c \leq \min(c_1, c_2)$. 
Conclusion (\ref{Dbig}) is an immediate consequence, after perhaps
making $c$ smaller.
\end{proof}

\note
{
{\em Fun fact. } Suppose  that condition (\ref{Dbig})
holds. Then we have $ \sobolev 2 {1/2} _D ( \partial \Omega )
\subset L^ 2 (\delta ^ {  -1} d\sigma) $ and hence
$ L^ 2( \delta \, d\sigma ) \subset \sobolev 2 {-1/2}  _D(\partial
\Omega)$.

Proof?

Example: Let $ \Omega = \{ x : | x| < 1\}$
and define $ D = \cup _{ k  = 1} ^ \infty
\{ (\cos \theta, \sin \theta ) : \theta \in ( 2^ { -2k} , 2^ {1 -2k}) \}$
and $ N = \partial \Omega \setminus D$. This domain
satisfies the conditions (\ref{SurfProp}) -- (\ref{NTA})
with $\epsilon =0$.}

The next two lemmas establish integrability of the distance function
$\delta$. 

\begin{lemma} \label{integrability}
Let $ \Omega$ and $D$ satisfy (\ref{Lip}), (\ref{SurfProp}), and
(\ref{NTA}). If $ x\in\partial \Omega$ and $r$ satisfies $ 0 < r <
r_0$, then for $-1+\epsilon < s< \infty$,
$$
\int _ { \sball x r } \delta (y ) ^ s \, d\sigma  \approx r ^ { n-1}
\max ( r, \delta (x) ) ^ s.
$$
.
\end{lemma}

\begin{proof}  Fix $ x\in \partial \Omega$, $r$ in the interval $
  (0, r_0)$ and consider several cases: a) $s\geq 0$ and $ \delta (x)
<  r /4$,
b) $s\geq 0$ and $ \delta (x)
\geq   r /4$,
 c) $s <  0$ and $ \delta (x)
< 4 r $, and
d) $s <  0$ and $ \delta (x)
\geq  4  r $.

{\em Case a)} Assume that $s\geq 0$, $\delta(x) < r/4$.  Since $ \delta (x) < r /
4$ and $ \delta $ is Lipschitz with constant one, we have $ \delta (y)
\leq 5r/4$ for all $ y \in \sball x r$. As $ \partial \Omega$ is the
boundary of a  Lipschitz domain, we have $ \sigma ( \sball x r)
\approx r ^ { n -1}$. Thus when $ s \geq 0$,  the upper bound $ \int _ { \sball x r }
\delta (y) ^ s \, d\sigma  \leq C r ^ { n -1+s}$  follows easily.

To obtain a lower bound, we begin by finding $ \hat x \in \Lambda$ such
that $ |x-\hat x| = \delta (x) < r /4$. Then the corkscrew condition
for $ D$ (\ref{NTA}) gives a point $\tilde x $ with $|x-\tilde x | < r
/4$ and $ \delta (\tilde x) > r /(4M)$.  If $ y \in \sball {\tilde x} { r/(8M)} $
we have that $ \delta ( y ) \geq r / ( 8M)$,  and  $ \sball { \tilde x} {
  r/(8M)} \subset \sball x r $ since $|x-\tilde x | < r /2$. Thus the lower bound
$ \int _ { \sball x r } \delta (y ) ^ s \,d\sigma \geq \int _ { \sball
 {\tilde  x}  { r/( 8M) } }  \delta ( y ) ^ s \, d\sigma \geq C r ^ {
  n -1+s}$ follows.
\note{
I think the argument for the lower bound can only work if we split the
cases at $ \theta r$  with $ \theta < 1$.
}

{\em Case b) } Assume that $s\geq 0$, $\delta(x) \geq  r/4$.
Observe that if $ y  \in \sball x r$, then
$ \delta (y) \leq \delta (x) +r\leq 5 \delta (x) $ and  the upper
bound
$\int _ { \sball x r } \delta (y) ^ s \, d\sigma \leq r ^ { n -1}
\delta ( x) ^ s $ follows.

To obtain the lower bound, we use that if $ \delta (x) \geq r / 4$ and
$ y \in \sball x { r/8}$, then $ \delta (x) \leq 2 \delta (y )$ and
hence we have $ \int _ {\sball x r  } \delta (y ) ^ s \, d \sigma \geq
  \int _ { \sball x { r/8 }} \delta ( y ) ^ s\, d\sigma \geq c r ^ {
      n-1} \delta(  x)  ^ s$.
\note{
$\delta (x ) \leq \delta ( y ) + r/8 = \delta ( y ) + r / ( 2\cdot
  4)\leq \delta(y )   + \delta (x) /2$.
}

{\em Case c) } Assume that $ -1+\epsilon < s < 0$,  $ \delta (x) < 4r$.  We divide the
surface ball $ \sball x r $ using level sets of the  distance function
and then use Lemma \ref{measure} to obtain
\begin{eqnarray*}
\int _ {\sball x r } \delta ( y ) ^ s d\sigma & = &
\sum _ { k = 0 } ^ \infty \int _ { \sball x r \cap \{ y : 2 ^ { -k-1}
  r < \delta (y) \leq  2^ { -k  }r \}}\delta (y ) ^ s \, d\sigma \\
& \leq & C  \sum _ { k = 0 } ^ \infty r ^ { n -2+\epsilon } ( 2 ^ { -k
}r)^ { s + 1 - \epsilon }\\
& \leq & Cr^ { n -1 + s}. 
\end{eqnarray*}
In the last inequality above we use the assumption that $ s > -1 + \epsilon $ 
to sum the geometric
series.

To obtain the lower bound, we observe that if $ \delta (x) \leq 4r$,
then for $y \in \sball x r $, $ \delta (y ) \leq \delta (x) + r <
5r$. Since $ s < 0$, it follows that $ \delta (y ) ^ s \geq ( 5r ) ^
s$ and the lower bound $ \int _ {\sball x r} \delta(y)^s\, d\sigma \geq C r^ { n -1 +
  s} $ follows easily.

{\em Case d) } Assume that $-1+\epsilon <  s < 0 $,  $\delta (x) \geq 4r$.  If $ y \in \sball
x r $, then we have
$ 3 \delta (x) /4 \leq \delta (y ) \leq 5 \delta ( x) /4$. Thus, we obtain
$\int _ {\sball x r } \delta (y ) ^s \, d\sigma \approx r ^ { n -1}
\delta (x) ^ s$.

\note{
And this case seems to require that we split at $Ar$ with $A>  1$.

The comparability follows since:
$$
\delta ( y ) \leq \delta ( x ) +r \leq \delta (x) + \delta (x) /4
$$
and
$$
\delta ( x ) \leq \delta ( y ) +r \leq \delta (y) + \delta (x) /4
$$
}

The result of the Lemma follows easily from the  four cases above.
\end{proof}

The integrability of the function $\delta$ over
interior balls is a straightforward adaptation of
the previous result.

\begin{lemma}\label{interiorint}
Let $\Omega$ and $D$ satisfy (\ref{Lip}), (\ref{SurfProp}), and
(\ref{NTA}), and let $r$ satisfy  $0<r<r_0$. Then for $s\in (-2+\epsilon, \infty)$, 
\begin{equation*}
\int_{\Psi_r (x)} \delta(y)^s \, dy \approx
r^ {n-2} \max (r, \delta(x))^s.
\end{equation*}
\end{lemma}

\note{
The previous lemma provides a condition on the exponent
$k$ such that $\delta(y)^k$ is an $A_{2/p}$ weight. We have,
$$
\left(\frac{1}{\sigma(\sball x r)} \int_{\sball x r}
\delta(y)^k\, d\sigma \right)\left( \frac{1}{\sigma(\sball x r)}
\int \delta(y)^{-k(p/(2-p))} \,d\sigma \right)^{(2-p)/p} \leq A < \infty,
$$
for all surface balls $\sball x r$ with $x\in \reals^{n-1}$
and $0<r<r_0$, under the restriction that
$$
k(p/(2-p))< 1-\epsilon\quad \mbox{or}\quad k<(1-\epsilon)((2-p)/p).
$$
We may use this with $k = 1-\rho$ in a future paper.
}

Throughout this work, the main tool for estimating solutions
will be the non-tangential maximal function. Fix $ \alpha >0$
and for  $ x\in \partial \Omega$,  the  {\em non-tangential
approach region } is defined by
$$
\Gamma(x) = \{ y \in \Omega : |x-y | \leq ( 1+ \alpha)
\dist (y, \partial\Omega) \}.
$$
Given a function $u$ defined on $ \Omega$,
the {\em  non-tangential maximal function } is defined as
$$
\nontan{u} (x) = \sup _{ y \in \ntar x } |u (y) |,
\qquad x \in \partial \Omega.
$$
We will also utilize a truncated non-tangential
approach region, 
$$
\Gamma_r (x) = \Gamma(x) \bigcap B_r(x),
$$
and, respectively, a truncated non-tangential
maximal function,
$$
\nontan{u}_r (x) = \sup_{y\in \Gamma_r (x)} |u(y)|,
\qquad x\in \partial\Omega.
$$
It is well known that for different values of
$ \alpha$, the non-tangential maximal functions have
comparable $L^p$-norms. Thus, we suppress the value
of $\alpha $ in our notation.

The restrictions of $u$ and $\nabla u$ to the boundary
in (\ref{MP}) are understood as non-tangential limits.
Precisely, for a function $v$ defined on $\Omega$ and
$x\in \partial\Omega$,
$$
v(x) = \lim _{ \Gamma (x) \ni y \rightarrow x } v(y),
$$
provided that the limit exists. It is well-known that for
a Lipschitz domain $\Omega$ and $v$ a harmonic function
in $\Omega$, the non-tangential limits exist at
almost every point where the non-tangential maximal
function is finite.

We now recall the definitions of atoms and atomic
Hardy spaces. A function $a$ is an \emph{atom for
$\partial\Omega$} if $\mbox{supp}\,a \subset \Delta_r (x)$
for some $x\in\partial\Omega$,
$\| a\|_{L^{\infty}(\partial\Omega)} \leq 1/\sigma
(\Delta_r (x))$
and $\int_{\partial\Omega} a \, d\sigma =0$. In our
treatment of the mixed problem, we will consider
atoms for the subset $N\subset \partial\Omega$. We
say that $\tilde{a}$ is \emph{an atom for $N$} if $\tilde{a}$
is the restriction to $N$ of an function $a$ which is
an atom for $\partial\Omega$. The Hardy space $H^{1}(N)$,
where $N\subset \partial\Omega$, is the collection of
functions $f$ which can be represented as
$\sum_j \lambda_j a_j$, where each $a_j$ is an atom for
$N$ and the coefficients  $\lambda_j$ satisfy
$\sum_j |\lambda_j| < \infty$. In the case where
$N=\partial\Omega$ this definition gives the standard
definition of the Hardy space $H^{1}(\partial\Omega)$.
The Hardy-Sobolev space $H^{1,1}(\partial\Omega)$ is
defined as the set of functions with one derivative
in $H^{1}(\partial\Omega)$. More precisely, we say
that a function $A$ is an \emph{atom for
$H^{1,1}(\partial\Omega)$} if $A$ is supported in a
surface ball $\Delta_r (x)$ for some $ x \in \partial \Omega$  and
$\| \tangrad A \|_{L^{\infty}(\partial\Omega)}
\leq 1/ \sigma(\Delta_r(x))$.  If $v$
is a smooth function
defined in a neighborhood of $\partial\Omega$
then the \emph{tangential gradient} of $v$ is defined as
$\tangrad v = \nabla v - (\nabla
v\cdot \nu )\nu$, where $\nu$ is the outward
unit normal vector.
Then $\tilde{A}$ is an \emph{atom for $H^{1,1}(D)$}
if $\tilde{A}$ is  the restriction to $D$ of an atom $A$ in
$H^{1,1}(\partial\Omega)$. The space $H^{1,1}(D)$
is the collection of all functions which can be
represented as $\sum_j \lambda_j A_j$ where each $A_j$ is
an element of $H^{1,1}(D)$ and $\sum_j |\lambda_j|<\infty$.

Finally, we define the Sobolev space $ \sobolev p1 (\partial \Omega)$
to be the collection of functions in $ L^ p ( \partial \Omega)$ whose
tangential gradient also lies in $ L^ p ( \partial \Omega)$.

\section{Green function estimates and reverse H\"older inequalities}

An important step in the proof of the main theorem is to
show decay of the solution to the mixed problem with
Neumann data an $H^1(N)$ atom as
we move away from the support of the atom. This decay is
encoded in estimates for the Green function for the mixed
problem which are proved in this section. The argument that
ensues only requires that
$ \Omega$ be a Lipschitz domain  and  that $D$ satisfies
(\ref{Dbig}).
\comment{
\begin{equation}\label{DBigLoc}
\sigma(\Delta_r (x) \cap D)\geq M^{-1}r^{n-1}.
\end{equation}
The condition above is the only assumption on the
boundary necessary to establish the results of this section.
}

When working near the boundary, we will want to assume
that part of the boundary is flat. This can always be
arranged in a Lipschitz domain by  flattening the boundary
with a change of coordinates. Since flattening the boundary
will change the coefficients,  we need to  consider operators $L$
with bounded and measurable coefficients. Assume
that $ L = \div A \nabla$, and assume that the coefficient matrix $A$
is real,   bounded, and
measurable, satisfies  $A^t=A$, and satisfies the ellipticity condition
that
for every  $\xi \in \reals^2$, there exists a $\lambda>0$ such that 
$$
\lambda |\xi |^2  \leq A\xi \cdot \xi \leq \lambda ^ { -1} | \xi |^ 2,
$$
The optimal 
$\lambda$ for which the above condition holds is called
the ellipticity constant for $L$. 

We now define a weak formulation of
the mixed problem for solutions of divergence form
operators whose gradients lie in $L^{2}(\Omega)$. Our
goal is to prove that under appropriate assumptions
on the data, the weak solution will have a gradient in
$L^{p}(\partial\Omega)$ for $1<p<p_0$, for some $p_0>1$.
For $k=1,2,\ldots$,  $\sobolev p k (\Omega)$ denotes
the Sobolev space of functions having $k$ derivatives
in $L^{p}(\Omega)$.  For $D$ an open subset of the
boundary, let $\sobolev 2 1 _{D}(\Omega)$ be the closure
in $\sobolev 2 1 (\Omega)$ of functions in
$C^{\infty}_{0}(\reals^n)$ whose support is disjoint from
$\bar{D}$.  Let
$\sobolev 2  {1/2} _D (\partial\Omega)$ be the restrictions to
$\partial\Omega$ of functions in  $W^{1,2}_D (\Omega)$ and define
$W^{-1/2,2}_D (\partial\Omega)$ to be the dual of
$W^{1/2,2}_{D} (\partial\Omega)$.
We assume that the Dirichlet data is zero and the Neumann data $f_N$ lies
in the space $W^{-1/2,2}_{D}(\partial\Omega)$.

Consider the mixed problem
\begin{equation} \label{MP3}
\left\{
\begin{array}{ll}
\mbox{div}A\nabla u = 0, \qquad & \mbox{in }  \Omega\\
u = 0,  \qquad &  \mbox{on }  D\\
A\nabla u \cdot \nu= f_N,\qquad  & \mbox{on } N.
\end{array}
\right.
\end{equation}
We say that $u$ is a \emph{weak solution} of the problem
(\ref{MP3}) if $u \in W^{1,2}_{D} (\Omega)$ and 
\begin{equation}\nonumber
\int_{\Omega} A\nabla u \cdot \nabla \phi \, dy =
\langle f_N, \phi \rangle, 
\quad \phi \in W^{1,2}_{D}(\Omega).
\end{equation}
%
To establish the existence of weak solutions to the
mixed problem, we assume the following coercivity
condition
\begin{equation}\label{coerce}
\|u \|_{L^{2}(\Omega)} \leq c \|\nabla u \|_{L^{2}(\Omega)},
 \quad u \in W^{1,2}_{D}(\Omega).
\end{equation}
Under this assumption, the existence and uniqueness of
weak solutions to the boundary value problem (\ref{MP3})
are a consequence of the Lax-Milgram theorem. In our
applications, $\Omega$ will be a connected, bounded
Lipschitz domain and $D$ will be an open subset of the
boundary. These assumptions are sufficient to ensure
that (\ref{coerce}) holds.

We also need to define a weak solution of the mixed problem
on a subset of $\Omega$. Let $\Omega '$ be an open subset
of $\Omega$. We say that $u$ is a weak solution to $Lu=f$
in $\Omega '$ with zero boundary data for the mixed problem on
$\partial\Omega' \cap \partial\Omega$ if
$u\in W^{1,2}_{\partial\Omega' \cap D} (\Omega')$ and for
each test function $\phi$ which lies in
$W^{1,2}_{\partial\Omega' \cap (D\cup \Omega)} (\Omega')$,
we have
$$
\int _ { \Omega ' } A \nabla u \cdot \nabla \phi \, dy = - \int_{\Omega'} f \phi
\, dy .
$$

\comment{
We fix $x\in \Omega $ and  $ \rho >0$ and define $G_\rho(x,\cdot)$ as the
solution of the mixed problem
$$
\left \{
\begin{array}{ll}
L G _ \rho (x, \cdot ) =  \frac {1} {| B_\rho | } \chi _ {\ball x \rho } (
\cdot), \qquad & \mbox{in } \Omega \\
G_\rho (x, \cdot ) =0, \qquad & \mbox{on } D\\
\frac {\partial  G_ \rho}{\partial \nu} (x, \cdot ) = 0, \qquad  & \mbox {on } N. \end{array}
\right.
$$
Thus, $ G _\rho (x,\cdot ) \in W^ { 1,2 }_D( \Omega)$ and
$$
\int _\Omega A \nabla G_\rho (x, \cdot ) \cdot \nabla \phi  \, dy
= - \frac 1 { |B_\rho | } \int _ {\ball x \rho } \phi\, dy , \qquad \phi \in
W^ { 1,2 }_D ( \Omega).
$$
We will prove estimates which are uniform in $ \rho$ and then use
these estimates to take the limit as  $\rho$ tends to zero.
}

Solutions of the mixed problem are bounded. If $ u$ is a
solution of $Lu =0$ in the domain $\dball x r =B_r (x)\cap \Omega$, $x\in \Omega$,
and $u$ has zero data for the mixed problem on
$ \partial \Omega \cap \partial   \dball x r $, then there exists a constant
$C>0$ such that
\begin{equation} \label{Moser}
|u(x) | \leq C \average_ { \dball x r}  |u(y) | \, dy .
\end{equation}
This may be proved by the Moser iteration method \cite{JM:1961}, for
example.

Finally, we give an estimate on the boundary H\"older continuity of
solutions of the mixed problem. For this estimate we consider domains
$ \dball x r = \ball x r $ with $ x \in \Omega$  and $ r < \dist(x, \partial
\Omega) $, or $ \dball x r = \ball x r \cap \Omega  $ with $ x \in \partial \Omega$
and $ r < r_0$. In the second case, we assume that $ \partial \Omega
\cap \ball xr $ lies in a hyperplane.
A study of H\"older continuity of solutions of elliptic equations may
be found in work of Stampacchia \cite{GS:1960}. Stampacchia gives a
general framework for studying H\"older continuity of
solutions using the method of De Giorgi \cite{EG:1957}. This framework is
applied to the mixed problem in the case where the
boundary of $D$ is the bi-Lipschitz image of a hyperplane of
co-dimension 2. Our assumptions allow for more general subdivisions of the
boundary.

\begin{theorem}\label{Holder} 
Let $ x\in \partial \Omega$ and assume 
that $  0 < r < r_0$. 
 Let $ u $ be a weak solution of the mixed problem in $   { \dball x r} $ 
 with zero data for the mixed problem on $ \partial
   {\dball x r} \cap \partial \Omega$. Then there exists an exponent
   $ \beta>0 $ such that
$$
| u(z) -u(y) |  \leq  C \left ( \frac { |z-y| } r \right ) ^ \beta
\sup _ { \dball x r } |u(y)| , \qquad z, y \in \dball x {r/2} .
$$
The constant $C$ and the exponent $\beta $ depend only on the
ellipticity constant $\lambda$ as well as  $M$ and
$n$. 
\end{theorem}

\def\cprime{$'$}
The proof here will follow the method of de Giorgi \cite{EG:1957} as
given in the monograph of Ladyzhenskaya and Ural\cprime tseva
\cite{MR0244627}. Fix $ \dball  {x_0} R$ as above. We say that a
bounded function  $u $ lies in the space
$  \luspace ( \dball {x_0} R, \gamma)$ if for each
$\dball x s  \subset \dball {x_0} R$, $ \sigma \in (0,1)$ and $k$
as below we have
$$
\int _ { A_{ k,s-\sigma s} } |\nabla u | ^ 2 \, dy \leq \frac { \gamma }{
  \sigma ^ 2 s^ 2 } \sup _{ A_ { k,s}} ( u -k ) ^ 2|A_{ k,s}|.
$$
Here, $ A_ { k,s} = \{ y \in \dball x s : u(y ) > k \}$ with $k$  an
arbitrary real number if $ \partial  \dball x s  \cap D = \emptyset$ and
$k \geq 0$ if $  \partial  \dball x s  \cap D \neq  \emptyset$.

\begin{lemma}
\label{Lemma0}
Let $x\in \partial\Omega$ and let $r$ satisfy
$0<r<r_0$.
If $ u$ is a solution  of the elliptic operator $L$ in
$ \dball x r$ with zero data for the mixed problem on
$ \partial \dball x r \cap \partial \Omega$, then
$ u \in \luspace ( \dball x r , \gamma)$ and
$ \gamma $ depends only on the ellipticity constant for $ L$.
\end{lemma}

\note{
We might consider omitting this proof as it is the same as in
Ladyzhenskaya.  If we do, we need to define $x^+$ before
}

\begin{proof}
Fix $ \dball y s$ which is contained in $ \dball x r$ and
$ \sigma \in (0,1)$. Let $\eta$ be a smooth cutoff function that
is supported in $ \ball y s $,  with
$ \eta =1$ on $\ball y { s-\sigma s}$ and satisfies $ |\nabla \eta |
\leq C_0 / ( \sigma s)$.   Let $u ^+ = \max(u,0)$  denote the
positive part of a function $u$.
If $k$ is as in the definition of the space $ \luspace$, then $ ( u-k
)^+\eta ^2$ may be used as a test function in the weak formulation of $
\div A \nabla  u =0$ and thus we have
$$
\int_ { \dball y s}  \eta ^ 2 A \nabla u \cdot \nabla  (u - k) ^ +\, dy
= -2 \int  _ { \dball y s} \eta
( u-k ) ^ +  A \nabla  u  \cdot  \nabla \eta \, dy .
$$
Using the symmetry and non-negativity  of $A$  and Young's inequality,
we obtain 
$$
2  \int _ { A_ { k,  s} } \eta^2 A \nabla u \cdot \nabla u  \, dy
\leq  \int_ { A_ {k,s}} \frac 1 2  \eta ^2 A \nabla u \cdot \nabla u  +
2  ( u -k ) ^ 2 A \nabla \eta \cdot \nabla \eta \, dy. 
$$
 Subtracting
the first term on  the  right and using the ellipticity of $A$ gives
$$
\frac \lambda 2 \int _ { A_ {k,s}}\eta^2 |\nabla u | ^ 2  \, dy \leq \frac 2 {\lambda } \int _ { A_
  {k,s}}( u-k ) ^ 2|\nabla \eta | ^ 2 \, dy .
$$
Recalling the estimate  $ |  \nabla \eta | \leq C_0 / ( \sigma s)$ and 
that $ \eta =1 $ on $ \ball y { s- \sigma s}$, we conclude that
%
$$
\int _ { A_ {k, s- \sigma s } } |\nabla u | ^ 2 \, dy
\leq \frac \gamma { ( \sigma s ) ^ 2} | A _{k,s} | \sup _{ A_ {k,s}}(
u-k)^2.
$$
Thus $u$ is in the space $\luspace (\dball x r , \gamma)$ with $
\gamma = 4 C_0^2/ \lambda ^2$.
 \end{proof}

\begin{lemma} \label{Lemma1}
Let $x\in \partial\Omega$ and assume that $0<r<r_0$.
If $u  \in \luspace (\dball x r , \gamma)$ and   $k$ is
as in the definition of this space, 
then there exists $ \theta _1=\theta_1(n,\gamma)$ such
that if $|A_{ k, r} | \leq \theta _1 r^n$ and $H = \sup_{ A_{ k,r} }
(u-k) > 0 $, then
$|A _ { k+ H/2, r/2} | =0$
and hence
$$
\sup _ { \dball x {r/2} } (u-k ) \leq H/2.
$$
\end{lemma}

Before giving the proof of Lemma \ref{Lemma1}, we need
to give two versions of the Sobolev-Poincar\'e inequality.

\begin{lemma}  \label{SPLemma}
Let $x\in \Omega$ and let $r$ satisfy $0<r<r_0$.
If $u \in W^ { 1,1 } ( \dball x r )$ and  $ \ell > k$, then 
\begin{equation}\label{SP2A}
( \ell -k ) |A_ { \ell,r} | ^ { 1- 1/n} \leq C \frac {| \dball x r |}
{| \dball x r \setminus A _ {k  ,r } |}\int _ { A_{ k,r} \setminus A_
  {\ell,r } }|\nabla u | \, dy .
\end{equation}

If  $ \ball x { r/2} \cap D \neq  \emptyset$,  $u \in W^ {1,1} ( \dball x
r)$, and $ k \geq 0$,
then
\begin{equation}\label{SP2B}
( \ell -k ) |A_ { \ell,r} | ^ { 1- 1/n} \leq C
\int _ { A_{ k,r} \setminus A_
  {\ell,r } }|\nabla u | \, dy .
\end{equation}
\end{lemma}

\begin{proof}
In each case, the estimate follows by applying a Sobolev inequality  to
the function
$$
v ( y ) = \left \{ \begin{array}{ll}
( u(y) -k ) ^+  , \qquad &  u (y ) \leq \ell  \\
\ell -k  , & u(y )  >  \ell
\end{array}
\right.
$$
The estimate (\ref{SP2B}) uses our assumption that $D$ satisfies
(\ref{Dbig}).  See Section  3  of \cite{OB:2009}, for example.
\end{proof}

\begin{proof}[Proof of Lemma \ref{Lemma1}]   Fix $\dball x r $ and let $u
  \in \luspace(\dball x r    , \gamma)$. For $h =0,1,2,\dots$, set
$$
r_h  =  r/2 + r/(2^ { h+1}) \qquad\mbox{and}\qquad
k_h  =    k + H/2 - H /  (2 ^ { h +1}),
$$
where $H$ is as in the statement of Lemma \ref{Lemma1}.
We will use the notation $\sigma _h = ( r_h - r _ { h+1} ) / r_h$. Since $ u
\in \luspace(\dball x r,\gamma)$, we have
\begin{equation} \label{eqnA}
\int _ { A_ { k _h , r _ { h+1} }}|\nabla u |^2 \, dy
\leq \gamma \frac 1 { ( r_ h - r _ {h +1} ) ^2} \sup_{ A _ { k_h, r _h
 }   } ( u - k_h  ) ^ 2
\leq \gamma  \frac { 2 ^ { 2h +4} } {r^ 2} H ^2 | A_ { k_h, r _h} |
.
\end{equation}
We use the inequality (\ref{SP2A}) from Lemma \ref{SPLemma} to obtain
that
\begin{equation} \label{eqnB}
(k_ { h+1}-k_h) | A_{ k_{h+1}, r _ { h+1}}| ^ { 1-1/n} \leq C \int _ {
  A_ { k _h, r_ { h+1}}}|\nabla u | \, dy,
\end{equation}
where we choose  $ \theta _1$  small in
order to obtain a uniform bound on the constant in the Sobolev
inequality in (\ref{SP2A}).
Now (\ref{eqnA}) and (\ref{eqnB}) give that
\begin{eqnarray*}
\frac H { 2 ^ { h +2}} | A_{ k _ { h+1}, r _ { h+1}}|^ { 1 - 1/n}
& \leq &  C \int _ { A_ { k _{h} , r _{h+1} } } |\nabla u | \, dy  \\
& \leq &  C \left ( \int _ { A_ { k_h, r_{h+1}}} |\nabla u | ^ 2 \, dy \right
) ^ { 1/2} | A_ { k _h, r _ {h +1}} | ^ { 1/2} \\
& \leq &  C \gamma ^ { 1/2 } 2 ^ { h + 2} \frac H r | A_ { k _h, r _ h } |
.
\end{eqnarray*}
Thus, we may conclude that
\begin{equation} \label{Recursion}
\left ( \frac { | A _ { k_ {h +1}, r _ { h + 1} }| } { r^ n }  \right)
^ { 1- 1/n} \leq C  \gamma ^ {1/2}
2 ^ { 2h + 4} \frac{ | A_ { k _h, r_h } | } { r ^ n } .
\end{equation}
According to Lemma 4.7 in  the monograph of Ladyzhenskaya and
\def\cprime{$'$}
Ural\cprime tseva
\cite[p.~66]{MR0244627}  if $ \theta _1$ is
sufficiently small, then the recursion relation (\ref{Recursion})
implies that
$
\lim _  { h \rightarrow \infty }  { |A_ { k _h, r _h } |}/{r^ n } = 0
.
$
\end{proof}

We now give the main step in the proof of H\"older continuity of
solutions of the mixed problem.   Before stating the result, we
introduce the notation $ \osc _E u = \sup _E u - \inf _E u $  for the {\em oscillation } of a
real-valued function
$u$ on a set $E$.

\begin{lemma}\label{MainStep}
Let $x\in\partial\Omega$ and assume that $0<r<r_0$.
Let $u$ be a solution of $Lu =0$ in $ \Psi _r (x)$ and  suppose that $u$ has
zero data for the mixed problem on $ \partial \Psi _r (x)  \cap
\partial \Omega$. Then there exists an integer $s$ such that
$$
\osc _ {\dball x r}  u  \leq ( 1 - 2 ^ { 1-s} ) \osc _{\dball x {4r} } u
.
$$
\end{lemma}

\begin{proof}
Since $\osc u = \osc (-u)$, it suffices to prove the Lemma for either
$u$ or $-u$. We will take advantage of this in the proof below.  We
define
\begin{eqnarray*}
M_r & =  & \sup \{ u(x): x \in \dball x r\} \\
m _ r & = & \inf \{ u(x) : x\in \dball x r \} \\
\omega _r&  =  & M_r - m _r = \osc _ { \dball x r } u \\
\bar M _r & = & \frac 1 2 ( M_r + m _r ).
\end{eqnarray*}
In what follows, set $\omega = \omega _ {4r} $ and
$$
D_t = A _ { M _ {4r} - \omega /2^ t , 2r} \setminus A _ { M _ { 4r} -
  \omega /2 ^ { t+1} , 2r } , \qquad t =1,2,\dots, s,
$$
where $s$ remains to be determined.

There are a few details that are different in the cases when
$\dball x {2r} \cap D
\neq
\emptyset$  and  $ \dball x { 2r} \cap D = \emptyset$ and we will
point out the differences when they arise.

In the case when $ \dball x {2r} \cap D = \emptyset $, we may assume
that
\begin{equation} \label{IntStart}
| A_ { \bar M _ { 4r} , 2r} | \leq \frac 1 2 |\dball x { 2r} |,
\end{equation}
for if the condition (\ref{IntStart}) fails, we may replace $u$
by $ -u$.
We next use the inequality (\ref{SP2A}) with $k = M _ { 4r} - \omega
/2^t$ and $ \ell = M _ {4r } - \omega / 2^ { t+1} $ to conclude that
\begin{equation}
\label{eqnC}
\frac \omega { 2 ^ { t+1} } | A _ { M _ { 4r} - \omega / 2 ^ { t+1},
  2r  } | ^ { 1-1/n} \leq C \int _ {D_t} |\nabla u | \, dy .
\end{equation}
Now since $D _ t \subset A _ { M _ { 4r}- \omega / 2 ^ { t} , 2r}$ and
$u \in \luspace ( \dball x r , \gamma) $, we have that
$$
\int _ {D_t} |\nabla u | ^ 2 \, dy \leq  \frac \gamma { 4 r^ 2}
|A_ { M_{ 4r}-\omega /2^t,     4r} |
 \sup  _
     { A_{ M_{4r} - \omega / 2^t , 4r }}  (u-(M_{4r}- \omega/2^t))^2
\leq C \left ( \frac \omega { 2^t} \right)^2 r ^ { n-2} .
$$
%
Thus, from (\ref{eqnC}) and the Cauchy-Schwarz inequality,  we have that
\begin{equation} \label{eqnD}
\left ( \frac \omega { 2 ^ { t+1}}\right )^2 |A _ { M _ {4r}- \omega
  /2^{t+1} , 2r } | ^ { 2 - 2/n}  \leq C |D_t| \left( \frac \omega { 2^ t}
\right) ^2 r ^ { n-2} .
\end{equation}
If we sum (\ref{eqnD}) from $ t =1 , \dots, s-3$, we conclude that
\begin{equation} \label{eqnF}
( s-3) |A _ { M _ { 4r}- \omega / 2 ^ { s-2}, 2r }| ^ { 2-2/n} \leq C r
    ^ { n-2} \sum _ { t=1} ^ { s-3} |D_t |
\leq C_0 r^ {  2n-2} .
\end{equation}
Choose $s$ such that
$$
\left ( \frac  { C_0 } { s-3} \right ) ^ { n / ( 2n -2)}\leq \theta
_1,
$$
with $ \theta _1$ as in Lemma \ref{Lemma1}.

Now let $ k = M _ { 4r} - \omega  / 2 ^ { s-2} $ and $H =
\sup  _ { \dball x { r}}( u - k) = M _ {r} - ( M _ { 4r} - \omega / 2 ^
     { s-2} ) $. If $H>0$, we may apply Lemma \ref{Lemma1} to obtain
$$
\sup _ { \dball x r } u \leq k + H /2
\leq M _ { 4r} - \omega / 2 ^ { s-2} + \frac 1 2 ( M _ r - M _ {4r} -
\omega / 2 ^ { s- 2} ).
$$
Simplifying, we find that
$$
M _ r \leq  M _  { 4r} - \omega / 2 ^ { s-1} .
$$
This inequality also follows easily if $H\leq 0$.
It is immediate to see that  $ -m_{ 4r } \geq - m _ r $ and if we recall
that  $ \omega =\omega _ {4r}$, we may conclude that
$$
\omega _r \leq ( 1 - 1 / 2 ^ { s-1} )\omega _ { 4r} .
$$

Next we consider the case when $  \dball x { 2r} \cap D \neq
 \emptyset$. In this situation we use the freedom to replace $u$ by $-u$ to impose
 the condition that $ \bar M _ {4r} \geq 0$ and, as a result,  (\ref{IntStart}) is
 not guaranteed to hold.

Since $ M _{4r} - \omega /2^t \geq \bar M _ {4r} \geq 0$, we may use the
Sobolev inequality  (\ref{SP2B}) to  obtain
$$
\frac \omega { 2 ^ { t+1}}| A _ { M _ { 4r}- \omega / 2 ^ t, 2r } | ^
      { 1-1/n} \leq C \int _ { D' _t} | \nabla u | \, dy,
$$
where $D'_t = A _ { M _ { 4r - \omega /2 ^ t , 4r }} \setminus A _ { M
  _ { 4r - \omega / 2 ^ { t+1}, 4r }} $. This  replaces (\ref{eqnC})
in the argument that  leads to (\ref{eqnF}).
The rest of the argument goes
  throughout without change.
\end{proof}

\begin{proof}[Proof of Theorem \ref{Holder}]
The Theorem follows immediately from Lemma \ref{MainStep}.
\end{proof}

We now return to working with only the Laplacian as this will simplify
the uniqueness argument below. 
We define a {\em Green function with pole at $x$  for the mixed
  problem} to be a function $G(x,\cdot) \in \sobolev 1 1 (
\Omega) $  which satisfies a) $ G(x, \cdot ) \in \sobolev 2 1 _D ( \Omega \setminus \ball x r )
$ for each $r>0$ and b) if $\phi $ is in $C^
\infty ( \bar \Omega)$ and vanishes on $D$, then
\begin{equation}  \label{DefProp}
\int _ \Omega \nabla G(x,\cdot ) \cdot \nabla  \phi  \, dy = - \phi  (x).
\end{equation}
The Green function is unique. If there are two
candidates for the Green function with pole at $x$, $G_1(x,\cdot ) $
and $G_2(x, \cdot) $, then $u = G_1(x, \cdot ) - G_2(x, \cdot) $ will
satisfy
$$
\int _ \Omega \nabla u \cdot \nabla \phi \, dy =0
$$
for all $ \phi \in C^ \infty ( \bar \Omega)$ which vanish on $D$. From
Weyl's lemma, $u$ is smooth in the interior of $ \Omega$. Then the
assumption that each $G_i ( x, \cdot) $ lies in $\sobolev 2 1_D ( \Omega
\setminus \ball x r )$ for each $ r >0$  implies that $ u $  is in
$\sobolev 2 1_D ( \Omega)$ Then $u$ is a weak solution of the mixed
problem with zero data and hence $u =0$.
The properties of the Green function for the mixed problem
that we will need in the sequel of this paper are summarized in the
following Lemma.

\begin{lemma} \label{Green}
Consider the mixed problem in a Lipschitz domain $\Omega$ with $D$
satisfying (\ref{Dbig}). Then there exists a Green function
for the mixed problem
 which satisfies: 1) If $G_x(y) = G(x,y)$,
then $ G_x \in\sobolev 2 1 _D ( \Omega \setminus \ball x r )$
for all $r>0$,  2) $\Delta G_x = \delta _x$, the Dirac
$\delta$-measure at $x$,   3) If $f_N \in \sobolev 2 {-1/2}  _D ( \partial \Omega)$, 
then  the weak  solution of the mixed problem
with $f_D=0$ can be represented by
$$
u ( x) =  - \langle f_N , G_x\rangle ,
$$
4) The Green  function  is H\"older continuous away from the pole and
satisfies the estimates
$$
|G(x,y) - G(x,y')| \leq \frac { C|y-y'|^ \beta } { |x-y |^ { n-2+\beta
  }} , \qquad |x-y| > 2 |y-y'|,
$$
$$
| G(x,y) | \leq \frac C { |x-y|^ { n-2} }, \qquad n\geq 3,
$$
and
$$
| G(x,y) | \leq C( 1+ \log  (d/|x-y|)) , \qquad n =  2.
$$
Above, the exponent $\beta$ is as in Theorem \ref{Holder}.
\end{lemma}

We give a detailed proof of this Lemma when $n=2$. The proof for $n
\geq 3$ may be obtained by adapting this argument or
by a straightforward adaptation of the arguments in Gr\"uter and
Widman \cite{MR657523}.

\begin{proof}
We begin with a result of Kenig and Ni \cite{MR87f:35065} who
established the existence of global fundamental solutions in two
dimensions.  Suppose that $L = \div A\nabla$ is an elliptic operator
in two dimensions with bounded, measurable, and symmetric coefficients.
Then there exists a fundamental solution $ \Gamma
(x,y)$ which satisfies
$
\Gamma (x, \cdot ) \in \sobolev  2 1 _{loc}( \reals ^ 2 \setminus \{
x\})$,  $\Gamma (x, \cdot )  \in \sobolev 1 1 _ {loc} ( \reals ^2)$,
and
we have
$$
\int_ { \reals ^ 2}   A \nabla \Gamma (x, \cdot  ) \cdot \nabla  \phi  \, dy
= -\phi (x), \qquad \mbox{for all}\,\,\,\phi \in C_ c ^ \infty( \reals ^ 2) .
$$
Chanillo and Li  \cite[Corollary 1.1]{MR1190215} observe that the free
space fundamental solution  lies in $ BMO( \reals ^2)$ and thus if $ \bar
\Gamma = \average_{\ball z s } \Gamma (x,y) \, dy $,  we have
$$
\average _{\ball z s } ( \Gamma(x,y ) - \bar \Gamma ) ^2 \, dy \leq C
$$
where the constant depends only on the ellipticity constant for the
operator $ L$ and the bounds for the coefficients.

\note {

It is not clear we need these pointwise estimates. We may go from BMO
to pointwise estimates after we construct the Green function for the
mixed problem.

In addition, there exists a constant $C$ and $0< R_1<1 < R_2$ such that
the Green function satisfies the pointwise estimates
$$
\Gamma (x,y) \leq C\log 1/ |x-y| , \qquad |x-y | < R_1
$$
and
$$
\Gamma (x,y) \leq C\log  |x-y| , \qquad |x-y | > R_2
$$
}

Next we recall a result of Dahlberg and Kenig \cite[p.~447]{DK:1987},
 if $ \Omega' = \{ ( x_1, x_2) : x_2 > \psi (x_1)\}$ is
the domain which lies above the graph of a single Lipschitz function,
then there exists a Green function for the Neumann problem in $
\Omega' $ which can be constructed by the method by reflection. We briefly recall the construction of this
Green function.  Define $R$ a reflection in $ \partial \Omega' $ by $
R(x) = ( x_ 1, 2\psi (x_1) - x_2) $ and construct an operator $ L =
\div A \nabla $ on $ \reals ^2$ so that $ L (u\circ R) =0 $ in $
\reals ^ 2 \setminus \bar \Omega' $ if and only if $ \Delta u =0 $ in
$ \Omega' $ and $ L= \Delta$ in $ \Omega '$.  
Let $ \Gamma$ be the fundamental solution for this
operator in $ \reals ^2$ and define
$$
\neugre(x,y) =  \Gamma(x,y) + \Gamma (Rx, y) .
$$
We  have that $\neugre$ is  a fundamental solution in $ \Omega '$ with
zero Neumann data on $ \partial \Omega' $. More precisely, we have the
weak formulation
\begin{equation}
\label{NeuDefProp}
\int _ {\Omega ' } \nabla \neugre (x,\cdot ) \cdot \nabla  \phi  \, dy  =
-\phi ( x) , \qquad \mbox{for all}\,\,\,\phi \in C_c^ \infty ( \bar \Omega).
\end{equation}

We give a detailed proof of the construction of the Green function for
the mixed problem $ G(x, \cdot)$ when $x$ is near the boundary and
hence lies in some coordinate cylinder.  When $x$ is far from the
boundary, the construction of the Green function is simpler and we
omit the details.
Fix $ x $ in our original domain $ \Omega$ and assume that $ x $
lies in a coordinate cylinder $Z_r ( x_0)$ with $Z_ { 4 r } ( x_0) $
also a coordinate cylinder.   Let $ s = \min (\dist ( x, D), r )$ and
then construct a cutoff function $ \eta $ which is one on $ \ball x
{s/2}$ and zero outside $\ball x s$. 
We let  $ \neugre  $ be the Green function for the Neumann  problem in a
graph domain $\Omega '$ that  satisfies
$\Omega \cap Z _ {4r} =  \Omega ' \cap Z_ { 4r}$. 
Since $ \Gamma$ and hence $
\neugre$ lie in $ BMO( \reals ^2)$, we may choose a constant $ \bar N$
so that 
\begin{equation}\label{BMO}
\average _{ \dball x s } ( \neugre ( x, y ) -\bar \neugre ) ^ 2 \, dy \leq C,
\end{equation}
where the bound $C$ depends only on the Lipschitz constant $M$.

 We will look for the Green
function for the mixed problem in the form
$$
G(x,y) = \eta (y) ( \neugre (x,y) -\bar \neugre) + u (y) .
$$
We will show that the   function $ u$  lies in $\sobolev 2 1 _D  ( \Omega)$ with a bound
depending only on the
Lipschitz constant $M$ and the  constant in (\ref{coerce}).

If $G$ is to be a Green function, we need $u$ to satisfy
$$
\int _{\Omega } ( \eta   \nabla   \neugre ( x,\cdot ) + (\neugre
(x,\cdot )-\bar \neugre )  \nabla \eta   +  \nabla u  ) \cdot \nabla
\phi    \, dy =  -  \phi(x)
$$
for all functions $\phi $ which lie in $ C^ \infty ( \bar \Omega) $ and vanish on $D$.
We write $ \eta \nabla \phi = \nabla ( \eta \phi ) - \phi \nabla \eta$
and get
$$
\int _ \Omega
  \nabla \neugre (x,\cdot ) \cdot \nabla ( \eta \phi)
-  \phi   \nabla \neugre (x,\cdot ) \cdot  \nabla \eta   +
 (\neugre (x,\cdot )-\bar \neugre )   \nabla \eta   \cdot \nabla \phi
+ \nabla u  \cdot \nabla \phi
\, dy = -  \phi(x).
$$
Since $ \neugre $ is a  Green  function for the Neumann problem, we
may use (\ref{NeuDefProp}) to simplify the previous equation and obtain
\begin{equation}  \label{udef}
\int _ \Omega     \nabla u  \cdot \nabla \phi    \, dy =
\int_\Omega  \phi     \nabla \neugre (x,\cdot) \cdot  \nabla \eta   -
 (\neugre (x,\cdot )-\bar \neugre )   \nabla \eta   \cdot \nabla \phi  \, dy .
\end{equation}
Let $F(\phi)$ denote the right-hand side of (\ref{udef}). We claim
that $|F(\phi)| \leq C\| \phi \| _ {\sobolev 2 1 _D( \Omega)} $ where
the constant depends only on the constant $M$ and the constant in the
coercivity estimate (\ref{coerce}).  From the claim and basic results about
Hilbert spaces, it follows that there exists a solution $u$ to
(\ref{udef}) and this solution satisfies
$
\|u\|_{ \sobolev 2 1  _D ( \Omega ) } \leq C.
$

We now turn to the proof of the claim.
To estimate the first term of (\ref{udef}), we begin with  
an application of the Cauchy-Schwarz inequality
$$
\left | \int_\Omega \phi   \nabla \neugre (x,\cdot) \cdot  \nabla \eta \,
dy \right |
 \leq C \left (\int  _{\dball x s \setminus \dball x {s/2} }  | \nabla
\neugre (x,\cdot ) | ^2 \, dy \right)^ { \frac{1}{2}}
\left( \average  _ {\dball x s } \phi ^2\, dy \right)^{\frac{1}{2}}.
$$
Using the Caccioppoli inequality and that
$ \neugre$ is  in $BMO(
\reals ^ 2)$  (see (\ref{BMO})), we may conclude 
$$
\left (\int  _{\dball x s \setminus \dball x {s/2} }  | \nabla
\neugre (x,\cdot ) | ^2 \, dy \right)^ { 1/2}  \leq C,
$$
where the constant depends only on the Lipschitz constant for $ \Omega$.
Thanks to the choice of $s$, we may use that $\phi $ vanishes on $D$
and (\ref{Dbig}) to obtain  the Poincar\'e  inequality
$$
\average _{\dball x s }   \phi  ^ 2 \, dy \leq C \int _{\dball x {Cs}}|\nabla \phi |^ 2\, dy .
$$
See Section 3 of Ott and Brown \cite{OB:2009} for details.

Thus we obtain the estimate
$$
\left |\int_\Omega  \phi   \nabla \neugre (x,\cdot ) \cdot  \nabla
\eta\, dy \right |  \leq C \|\nabla \phi \|_{L^2 ( \Omega)}.
$$
The estimate for the other term,
$$
\left |\int _{\Omega}  (\neugre (x,\cdot )-\bar \neugre )   \nabla
\eta   \cdot \nabla \phi    \, dy
\right |
\leq C \|\nabla \phi \|_ { L^ 2 ( \Omega)},
$$
follows from the Cauchy Schwarz inequality since $ \neugre(x, \cdot) $
is in $BMO(\reals^2)$.

Next we recall that if $u$ is in $\sobolev 2 1  ( \Omega )$, $x \in
\Omega$, and $r > 0 $, we may find a constant $ \bar u$ such that
we have the Poincar\'e inequality
$$
\average  _ { \dball x r } ( u -\bar u ) ^ 2 \, dy \leq C \|\nabla u
\|_{L^2 ( \Omega ) } ^ 2.
$$
In other words, $u$ lies in $BMO( \Omega)$.
Since we also have that $\neugre $ is in $BMO ( \reals ^ 2)$, we may
conclude that $ G$ lies in $BMO( \Omega)$.

Now we turn to the estimates in part 3) and 4) of  Lemma \ref{Green}.
First, recall that if
$v$ is in $BMO( \Omega)$ and $ \dball  x r \cap \dball {x' } { 2r }
\neq \emptyset$, then we have
$$
\left | \average _{ \dball x r } v \, dy -  \average _{ \dball {x'} {
    2r}} v\, dy \right | \leq \|v\| _*,
$$
where $ \|v\|_*$ is the $BMO$ norm of $v$.
Using this, an iteration argument  and the local boundedness estimate
(\ref{Moser}), we obtain the
pointwise upper bound,
$$
|G(x,y )| \leq C ( 1 + \log ( d / |x-y|)),
$$
where $d$ is the diameter of $ \Omega$ and the constant $C$ depends on
$\Omega$.
\marginpar{Specify dependence of constant more precisely.}
Next we show that the Green function is H\"older continuous.
Let $v$ be a solution of $ Lv=0$ in $ \dball x r$ with zero data for
the mixed problem on $ \partial \Omega \cap \partial \dball x r$.
From the local boundedness result (\ref{Moser}) and the estimate for
H\"older continuity in Theorem \ref{Holder},
for any constant $ \bar v$ we have
$$
|v(y) -v(y') | \leq C (|y-y ' | /r)^ \alpha  \average  _{\dball x r} |v-\bar v |
\, dy, \qquad y, y' \in \dball x {r/2}   .
$$
Since  $G$ is in $BMO( \Omega)$, the H\"older estimate for $G$ in
part 4) of Lemma \ref{Green} follows by applying the above observation
to $ G(x,\cdot)$ on a ball centered at $y$ with radius $ r = |x-y|$.

Next, we claim that if $ f_N \in \sobolev 2 {-1/2} _D ( \partial
\Omega)$ and $u$ is the weak solution of the mixed problem with
Neumann data $f_N$ and zero Dirichlet data, then we have the
representation formula in part 3) of Lemma \ref{Green}
\begin{equation}
\label{RepFmla}
u(x) =  - \langle  f_N  , G_x  \rangle.
\end{equation}
Here,  $ G_x = G(x,\cdot)$ and $ \langle \cdot, \cdot \rangle $
denotes the duality pairing between  $ \sobolev 2 {
-1/2} _D ( \partial \Omega)$ and $ \sobolev 2 {1/2} _D( \partial
\Omega)$.  To begin the proof of (\ref{RepFmla}),  consider $ \int_ \Omega \nabla G(x,
\cdot ) \cdot \nabla u \, dy.  $ Fix $ x\in \Omega$, let $r = \dist
( x , \partial \Omega)/2$, and let $ \eta $ be a cutoff function with $
\eta = 1 $ on $ \ball x { r/2}$ and $ \eta = 0 $ outside $\ball x r$.  We
may approximate $\eta G(x, \cdot)$ in $ \sobolev 11 ( \Omega)$ by a
sequence of smooth functions and use that $u $ is harmonic in $\ball x
r $ to conclude that $
\int _ \Omega \nabla ( \eta G(x, \cdot ) )\cdot \nabla u \, dy = 0.  $
Since $ ( 1- \eta ) G(x, \cdot)$ lies in $ \sobolev 2 1 _D ( \Omega)$,
we may use that $u$ is a weak solution of the mixed problem to
conclude that $ \int_{ \Omega } \nabla (( 1-\eta ) G(x, \cdot )) \cdot
\nabla u \, dy = \langle f_N , G_x\rangle $. Combining these
observations gives
\begin{equation}\label{RepFmlapt1}
\int_ \Omega \nabla G(x, \cdot ) \cdot \nabla u \, dy = \langle f_N ,
G_x \rangle .
\end{equation}

We now reverse the roles of $G$ and $u$. With $ \eta $ as above, 
write $u = \eta u + ( 1-\eta ) u$. As $u $ is  harmonic and hence
smooth in the interior of $ \Omega$, there exists a sequence of
smooth, compactly supported  functions which
converge to $ \eta u$ in $\sobolev     \infty  1 ( \Omega)$.  We may use
this sequence and (\ref{DefProp}) to obtain
$\int _ \Omega \nabla G \cdot \nabla ( \eta u ) \, dy = -u(x)$.  As $
u $ lies in $ \sobolev 2 1 _D( \Omega)$ we may  find a sequence
$\{u_k\}$ of
smooth functions which vanish on $D$ and which converge in $\sobolev 2
1 (\Omega)$ to $u$. Approximating $ ( 1-\eta ) u$ by $ ( 1- \eta ) u
_k$ and using (\ref{DefProp})  gives $ \int_\Omega \nabla
G(x, \cdot)\cdot \nabla ((1-\eta)u) \, dy  = 0$. Combining these observations
implies
\begin{equation}
\label{RepFmlapt2}
\int_ \Omega \nabla G(x, \cdot ) \cdot \nabla u \, dy =  - u(x).
\end{equation}
From (\ref{RepFmlapt1}) and (\ref{RepFmlapt2}), we obtain
(\ref{RepFmla}).

\end{proof}

The next two lemmas establish higher integrability of the gradient
of weak solutions to the mixed problem. The proofs of these lemmas
appear in Ott and Brown \cite[Section ~3]{OB:2009}.
The key ingredients of the proofs appearing in the aforementioned paper are
Poincar\'{e} inequalities, and these Poincar\'{e} inequalities
continue to hold true in the current setting due to the condition
(\ref{Dbig}) on $D$. 
A similar estimate  is obtained for the mixed problem by Gr\"oger
\cite{MR990595} using the method of N. Meyers \cite{NM:1963}. However,
Gr\"oger's method requires more restrictions on the boundary between
$D$ and $N$.

\begin{lemma}\label{MSIRHI}
Let $\Omega$ and $D$ satisfy (\ref{Lip}) and (\ref{Dbig}).
Let $x\in\Omega$ and let $r$ satisfy $0<r<r_0$.
Let $u$ be a weak solution of the mixed problem for a
divergence form elliptic operator with zero Dirichlet data and
Neumann data $f_N \in L^{q}(N)$. Then $u$ satisfies the estimate
$$
\left ( \average _{ \dball x  { r} }  |\nabla u | ^ 2 \, dy \right ) ^
  { 1/2 }  \leq C \left [
\average _ {\dball x  {Cr} } |\nabla u | \, dy
+\left (  \frac 1 { r^ { n-1}} \int  _ { \sball x {Cr} }
|f_N|^ { q} \, d\sigma \right) ^ { 1/q}
\right  ] .
$$
Above, $q=2$ if $n=2$ and $q=2(n-1)/(n-2)$ for $n\geq 3$.
The constant $C$ depends on $M$ and $n$.
\end{lemma}

\begin{lemma}
\label{RHEstimate}
Let $\Omega$ and $D$ satisfy (\ref{Lip}) and (\ref{Dbig}).
Let $x\in\Omega$ and let $r$ satisfy $0<r<r_0$.
Let $u$ be a weak solution of the mixed problem with zero
Dirichlet data and Neumann data $f_N \in L^q(N)$ which is
supported in $ N \cap \sball x r $, with $q$ as in Lemma
\ref{MSIRHI}.
Then there exists  $ q_0=q_0(M,n) > 2 $ such that for $t $ in the range
$2\leq t < q_0$ when $n\geq 3$ and $t$ in the range $2<t<q_0$ when $n=2$,
$u$ satisfies the estimate
\begin{eqnarray*}
\lefteqn{  \left (  \average _{ \dball x r  } |\nabla u |^t \, dy  \right
)^ { 1/t}  } \\
& \leq  &   C\left[   \average _ {\dball  x {Cr} } |\nabla u  |\,dy
  +\left( \frac 1 { r^ { n-1} }
   \int  _{ \sball x {Cr}  }  |f_N|^{t(n-1)/n}\,
d\sigma\right) ^ { n/(t(n-1))}\right] .
\end{eqnarray*}
The constant above depends on $M$ and $n$.
\end{lemma}

\section{Estimates for solutions with atomic data} \label{Sect3}

In this section we consider the mixed problem with
Neumann data an atom for $N$ and zero Dirichlet data.
We estimate the decay of the solution of this mixed problem
as we move away
from the support of the atom by taking $L^p$-norms of
the solution in dyadic  rings around the support of the atom.
Thus, given a surface ball $ \sball x r$,  we  define
$\Sigma_k = \sball x { 2^ k r} \setminus \sball x {2^ { k-1}
  r}$  and    $ S_k =   \dball x {2^k r } \setminus \dball
  x { 2^ {    k-1} r }  $.

\begin{theorem} \label{AtomicTheorem}
Let $\Omega$  and $D$ satisfy (\ref{Lip})
and (\ref{NTA}). Fix $x\in\partial\Omega$ and let $r$
satisfy $0<r<r_0$.
Let $u $ be a weak solution of the mixed problem
(\ref{MP}) with Neumann data $f_N=a$ an atom for $N$  which
is supported in $\sball x r$ and zero Dirichlet data. Let
$q_0>2$ be as in Lemma \ref{RHEstimate} and let $\Lambda$ satisfy
(\ref{SurfProp}) with $\epsilon$ such that
$0\leq \epsilon < (q_0-2)/(q_0-1)$. Then for
$ 1 < p < q_0((1-\epsilon)/(2-\epsilon))$, the
following estimates hold
\begin{equation} \label{LocalPart}
\left( \int _{\sball x {8r} }   |\nabla u |^p \, d \sigma  \right)^ {
  1/p}
\leq C \sigma (\sball x  {r} )^ {-1/p'},
\end{equation}
and   for $ k \geq 4$,
\begin{equation} \label{Decay}
\left ( \int _{ \sring k}  |\nabla u |^p \,d\sigma  \right) ^ { 1/p}
\leq
C 2^ {-\beta k}  \sigma(  \sring k ) ^ {- 1/p'} .
\end{equation}
Here,  $\beta $ is as in Lemma \ref{Green} and the constants in
the estimates (\ref{LocalPart}) and (\ref{Decay}) depend on
$p$ and the global character of the domain.
\end{theorem}

In order to prove Theorem \ref{AtomicTheorem} we need
a series of lemmas, some of which require that we work in a
subdomain of $\Omega$ which is also contained in a coordinate
cylinder.
 For $x\in \partial\Omega$ and $r$ satisfying $0<r<r_0$,
 let $\locdom x r  = Z_r (x) \cap \Omega$,
where $Z_r (x)$ is a coordinate cylinder as  defined in Section
\ref{prelim}.  The sets $ \locdom xr $ are star-shaped Lipschitz
domains and for this reason they are preferable to the sets $ \dball x
r$.

The following lemmas rely on a Whitney decomposition of
$\partial\Omega \setminus \Lambda$. For simplicity, we use
surface cubes rather than the surface balls used up to this
point. A \emph{surface cube} is the image of a cube in
$\reals^{n-1}$ under the mapping $x'\rightarrow(x',\phi(x'))$.
Then we write $\partial\Omega = \Lambda \bigcup (\cup_j Q_j)$
where the collection of surface cubes $\{Q_j\}$ has the following
three properties: 1) For each $j$, either $Q_j\subset D$ or
$Q_j \subset (N\setminus \Lambda)$, 2) There exist constants $c'$ and  $c''$,
with $c''$ as small as we like, such that for each $x\in Q$ and
each $j$,
$c'\delta(x) < \mbox{diam}(Q_j) <  c'' \delta(x)$, 3) If 
$T(Q_j) = \{\, x\in\bar{\Omega}: \mbox{dist}(x,\partial\Omega)< \mbox{diam}(Q_j)\,\}$,
then provided that the constants in the previous condition are
sufficiently small, the sets $\{T(Q_j)\}$ have bounded overlaps
and thus
\begin{equation*}
\sum_j \chi_{T(Q_j)} \leq C(n,M).
\end{equation*}

We now begin the series of lemmas. The first two lemmas
give a local version of boundary regularity for the Dirichlet
and Neumann problems for the Laplacian, and they require only that
$\Omega$ be a Lipschitz domain. The proofs of the next two lemmas
appear in the previous work of Ott and Brown
\cite[Section ~4]{OB:2009}.

\begin{lemma}\label{NeumannRegularity}
Let $\Omega$ be a Lipschitz domain, let $x\in \partial\Omega$, and
assume that $r$ satisfies $0<r<r_0$.
Let $u$ be a harmonic function in $\Omega_{4r}(x)$.
If
$\nabla u \in L^{2}(\Omega_{4r}(x))$ and
$\partial u/\partial \nu \in L^{2}(\partial\Omega \cap \partial\Omega_{4r})$,
then $\nabla u \in L^{2}(\sball x r)$ and 
$$
\int _ { \sball x {r}}  (\nontan {( \nabla u )}_{r})^2 \, d\sigma
\leq C \left ( \int _ { \partial \Omega \cap \partial \locdom x {4r}  } \left |\frac { \partial u  }{ \partial
    \nu } \right | ^ 2 \, d\sigma
+ \frac 1 r  \int _ {\locdom x {4r}  } |\nabla u |^2 \,
dy\right).
$$
The constant $C$ depends on $M$ and the dimension $n$.
\end{lemma}

\begin{lemma}\label{DirichletRegularity}
Let $\Omega$ be a Lipschitz domain, let $x\in \partial\Omega$ and
$r$ such that $0<r<r_0$.
Let $u$ be a harmonic function in $\Omega_{4r}(x)$.
If $\nabla u  \in L^2 (\locdom x {4r} )$ and
$ \tangrad  u \in L^ 2 (\partial \Omega \cap \partial \locdom x { 4r} )$, then
$ \nabla u  \in L^2 ( \sball x r)$ and 
$$
\int _ { \sball x {r}}  (\nontan {( \nabla u )}_r)^2 \, d\sigma
\leq C \left( \int _ { \partial \Omega \cap \partial  \locdom x { 4r} }| \tangrad u |^2 \, d\sigma
+  \frac 1 r \int _ {\locdom x { 4r} }  |\nabla u |^2 \,
dy\right).
$$
The constant $C$ depends on $M$ and the dimension $n$.
\end{lemma}

The next lemma also appears in Ott and Brown \cite[Section 4]{OB:2009}.
For the sake of completeness we sketch the proof again in
this paper.

\begin{lemma}\label{Whitney}
Let $\Omega$  and $D$ satisfy (\ref{Lip}), (\ref{SurfProp}),
and (\ref{NTA}). Let $u$ be a weak solution of the mixed problem with
Neumann data $f_N \in L^2(N)$ and zero Dirichlet data. Let $\rho \in\reals$,
$x\in\partial\Omega$, and $0<r<r_0$, and assume that for some $A>0$,
$\delta(x)\leq Ar$. Then it follows that
$$
\int _ {\sball x r    }   (\nontan{( \nabla u)}_{c\delta}) ^2 \, \delta  ^{1-\rho}  d\sigma
\leq C \left ( \int _ { \sball x {2r}}
|f_N|^2  \delta  ^ { 1- \rho } \,  d\sigma
  + \int _ { \dball x { 2r}   }  |\nabla u  |^2 \, \delta  ^ {
    -\rho } \, dy
\right ),
$$
for constants $c$ and $C$ which depend only on $M$, $n$, $A$ and $
\rho$.
\end{lemma}

\begin{proof}
When $\Delta_r(x)$ is close
to $\Lambda$, using the Whitney decomposition constructed
above and the estimates of Lemma \ref{NeumannRegularity}
and \ref{DirichletRegularity}, we have
\begin{equation}\label{Whitneyeq}
\int_{Q_j} ((\nabla u)_{r_j}^{*})^2 \,d\sigma \leq
C\left(\int_{2Q_j \cap N} \left| \frac{\partial u}{\partial \nu} \right| ^2
\, d\sigma + \frac{1}{r_j}\int_{T(Q_j)} |\nabla u|^2 \, dy \right).
\end{equation}
To finish the proof of the Lemma, we multiply (\ref{Whitneyeq})
by $r_j^{-\rho}$, recall that $r_j \approx \delta(x)$ for all
$x\in T(Q_j)$, sum on the $Q_j$ that intersect $\sball x r$,
and use that the family
$\{T(Q_j)\}$ has bounded overlaps.
\end{proof}

The next result is another reverse H\"older inequality,
this time at the boundary. While at first glance the result
below may not resemble a reverse H\"older inequality,
in future applications in this paper $f_N=0$ or
a constant.

\begin{theorem}\label{RHBoundary}
Let $\Omega$ and $D$ satisfy (\ref{Lip}) and (\ref{NTA}).
Let $q_0>2$ be as in Lemma \ref{RHEstimate} and let $\Lambda$ satisfy
(\ref{SurfProp}) with $0\leq \epsilon < (q_0-2)/(q_0-1)$. Let $u$ be
the weak solution of
the mixed problem with Neumann data $f_N \in L^2(N)$ and zero
Dirichlet data.  Fix $p$ such that $1< p < q_0 (1-\epsilon)/(2-\epsilon)$.
For $x\in\partial\Omega$ and $r$ satisfying $0<r<r_0$, 
\begin{equation}\label{result62}
\left(\average_{\sball x {r}} |\nabla u |^p \, d\sigma \right)^{1/p}
\leq C \left( \average_{\dball x {2r}} |\nabla u| \, dy + r^{(1-n)/2}\|f_N \|_{L^2 (\sball
x {2r} \cap N)} \right).
\end{equation}
%
%
Above, the constant $C$ depends on $M$, the dimension $n$, and $p$.
\end{theorem}

\begin{proof}
Fix $x\in \partial\Omega$ and $0<r<r_0$. We claim that 
\begin{equation}\label{claim62}
\left(\average_{\sball x {4r}} |\nabla u |^p \, d\sigma \right)^{1/p}
\leq C \left( \average_{\dball x {16r}} |\nabla u| \, dy + r^{(1-n)/2}\|f_N \|_{L^2 (\sball
x {32r} \cap N)} \right).
\end{equation}
We will separate the proof into two cases: a) $\delta(x) \leq 8r\sqrt{1+M^2}$, and
b) $\delta(x) > 8r\sqrt{1+M^2}$. Starting with case a), choose $\rho$
satisfying $2-\frac{2}{p}(1-\epsilon)-\epsilon < \rho <
2-\frac{4}{q_0}(1-\epsilon) - \epsilon$ (the
assumption that $0\leq \epsilon < (q_0-2)/(q_0-1)$ ensures that
this is a non-empty interval). Apply H\"older's inequality
with the exponents $2/p$ and $2/(2-p)$ to get
\begin{eqnarray}
\lefteqn {\left(\int_{\sball x {4r}} |\nabla u|^p\, d\sigma\right)^{1/p} } \qquad \nonumber \\
 &\leq & \left(\int_{\sball x {4r}} |\nabla u|^2 \delta^{1-\rho}\, d\sigma\right)^{\frac{1}{2}}
\left(\int_{\sball x {4r}} \delta^{(\rho-1)\frac{p}{2-p}} \, d\sigma \right)^{\frac{1}{p} -
\frac{1}{2}} \nonumber \\
&\leq &Cr^{(n-1)(\frac{1}{p}-\frac{1}{2})+\frac{\rho-1}{2}} \left(\int_{\sball x {4r}}
|\nabla u|^2 \delta^{1-\rho}\, d\sigma\right)^{1/2}, \label{eqn416}
\end{eqnarray}
where we have used that $(\rho-1)p/(2-p) < -( 1-\epsilon)$ or $\rho >
2-\frac{2}{p}(1-\epsilon)-\epsilon$, and Lemma \ref{integrability}
to ensure that the integral of $\delta^{(\rho-1)(p/(2-p))}$ is finite.
Next, we use Lemma 4.6 and our hypothesis that $\delta(x) \leq 8r\sqrt{1+M^2}$ to obtain
\begin{eqnarray}\label{eqn05}
\lefteqn{ \left(\int_{\sball x {4r}} |\nabla u|^2 \delta^{1-\rho}\, d\sigma \right)^{1/2} } \qquad \nonumber \\
& \leq & C \left[ \left(\int_{\dball x {8r}} |\nabla u|^2 \delta^{-\rho} \, dy \right)^{1/2}
+ \left(\int_{\sball x {8r} \cap N} |f_N|^2 \delta^{1-\rho}\, d\sigma \right)^{1/2}\right]\qquad\qquad \nonumber \\
& \leq & C \left[\left( \int_{\dball x {8r}} |\nabla u|^2 \delta^{-\rho}\, dy\right)^{1/2} +
r^{\frac{n-\rho}{2}} \|f_N\|_{L^2(N\cap \sball x {8r})}\right].
\end{eqnarray}

To estimate the term $(\int_{\dball x {8r}} |\nabla u|^2 \delta^{-\rho}\,dy )^{1/2}$,
choose $q>2$ such that $q<\min\{q_0,2n/(n-1)\}$ and apply H\"older's
inequality again with exponents $q/2$ and $q/(q-2)$ to conclude
\begin{eqnarray*}
\lefteqn{ \left(\int_{\dball x {8r}} |\nabla u|^2 \delta^{-\rho}\, dy \right)^{1/2} } \qquad\qquad \nonumber \\
& \leq & \left(\int_{\dball x {8r}} |\nabla u|^q \, dy \right)^{1/q} \left(\int_{\dball x {8r}}
\delta^{-\rho \frac{ q}{q-2}} \, dy\right)^{\frac{1}{2}-\frac{1}{q}}.
\end{eqnarray*}
We invoke Lemma \ref{RHEstimate} again, requiring that $q\in (2,q_0)$, and
get the following bound
\begin{eqnarray}
\lefteqn{ \left(\int_{\dball x {8r}} |\nabla u|^2 \delta^{-\rho}\, dy
  \right)^{1/2}  }    & &    \label{eqn07}  \\
& \leq & Cr^{\frac{n-\rho}{2}} \left[ \average_{\dball x {16r}} |\nabla u|\, dy +
\left( \frac{1}{r^{n-1}} \int_{\sball x {16r}\cap N}
|f_N|^{\frac{q(n-1)}{n}}\, d\sigma\right)^\frac{n}{q(n-1)}\right].
\nonumber
\end{eqnarray}
By our choice of $q$, we can apply H\"older's inequality with
exponent $2n/(q(n-1))$ to the boundary term  in (\ref{eqn07}) to obtain
\begin{eqnarray}\label{eqn08}
\lefteqn{
\left(\frac{1}{r^{n-1}} \int_{\sball x {16r}\cap N}
|f_N|^{\frac{q(n-1)}{n}} \, d\sigma \right)^{\frac{n}{q(n-1)}}
} \hspace{1in}
 \\
\nonumber
&\leq &  C\left(\frac{1}{r^{n-1}} \int_{\sball x {16r}\cap N} |f_N|^2\, d\sigma\right)^{1/2}.
\end{eqnarray}
Combining equations (\ref{eqn416}), (\ref{eqn05}), (\ref{eqn07}), and (\ref{eqn08}) we conclude that
\begin{equation*}
\left(\int_{\sball x {4r}} |\nabla u|^p\, d\sigma\right)^{\frac{1}{p}} \leq
Cr^{\frac{n-1}{p}}\left(\average_{\dball x {16r}} |\nabla u|\, dy + r^{(1-n)/2}\|f_N\|_{L^2(\sball x {16r} \cap N)}\right),
\end{equation*}
which gives the claim (\ref{claim62}).

Now we prove the claim (\ref{claim62}) under the condition in
case b). As with case a), we begin with an application
of H\"olders inequality
\begin{equation*}
\left( \int_{\sball x {4r}} |\nabla u |^p\, d\sigma\right)^{1/p}
\leq Cr^{(n-1)(1/p-1/2)} \left(\int_{\sball x {4r}} |\nabla u|^2 \, d\sigma\right)^{1/2}.
\end{equation*}
Use that $\delta(x)>8r\sqrt{1+M^2}$
and Lemma \ref{newball} to conclude that either $\sball x {8r} \subset N$ or
$\sball x {8r} \subset D$. Then we may appeal to Lemma \ref{NeumannRegularity} or Lemma
\ref{DirichletRegularity}
to obtain that
\begin{equation*}
\int_{\sball x {4r}} |\nabla u|^2 \, d\sigma \leq
C\left(\int_{\sball x {8r} \cap N } |f_N|^2 \, d\sigma + \frac{1}{r} \int_{\dball x {8r}} |\nabla u|^2\, dy\right).
\end{equation*}
From this point we must distinguish between $n=2$ and $n\geq 3$. First,
let $n\geq 3$. Then Lemma \ref{RHEstimate} and H\"older's inequality give that
\begin{equation*}
\left(\average_{\dball x {8r}} |\nabla u|^2 \, dy \right)^{\frac{1}{2}} \leq
C \left[\average_{\dball x {16r}} |\nabla u|\, dy
+ \left(\int_{\sball x {16r}\cap N} |f_N|^{\frac{2(n-1)}{n}} \, d\sigma\right)^{\frac{n}{2(n-1)}}\right].
\end{equation*}
Combining the last three displayed inequalities, we obtain
\begin{equation*}
\left(\int_{\sball x {4r}} |\nabla u|^p\, d\sigma\right)^{\frac{1}{p}}
\leq C r^{\frac{n-1}{p}} \left( \average_{\dball x {16r}} |\nabla u|\, dy
+ r^{(1-n)/2} \|f_N\|_{L^2(\sball x {16r} \cap N)}\right),
\end{equation*}
which leads immediately to the claim (\ref{claim62}). In the case $n=2$, we need
an additional application of H\"older's inequality. Choose $t$ satisfying
$2<t<q_0$. Then $(\average_{\dball x {16r}}|\nabla u|^2\, dy)^{1/2} \leq C
(\average_{\dball x {16r}} |\nabla u|^t \, dy)^{1/t}$ and from this point
we can now apply Lemma \ref{RHEstimate} to get the average of the
square of the Neumann data. This gives (\ref{claim62}).

Once (\ref{claim62}) is established, an
elementary covering argument leads to the desired estimate.
\end{proof}

The final result we require before proving Theorem \ref{AtomicTheorem}
is an energy estimate. The proof appears in Ott and Brown
\cite[Section 4]{OB:2009}.

\begin{lemma}\label{Energy}
Let $u$ be a weak solution of the mixed problem with Neumann
data $f_N$. For $n\geq 3$, let $f_N\in L^{p}(N)$ with
$p=(2n-2)/n$. Then the following estimate holds
$$
\int_{\Omega} |\nabla u|^2 \, dy \leq C \|f_N\|^{2}_{L^{p}(N)}.
$$
If $n=2$, let $f_N\in H^1(N)$ and then the following estimate
holds
$$
\int_{\Omega} |\nabla u|^2 \, dy \leq C\|f_N\|^{2}_{H^{1}(N)}.
$$
In both cases, the constant $C$ depends on  the global character of
$\Omega$. 
\end{lemma}

We are now equipped to prove Theorem \ref{AtomicTheorem}.

\begin{proof}[Proof of Theorem \ref{AtomicTheorem}]

Fix $x\in \partial\Omega$ and let $r$ satisfy $0<r<r_0$.
The first step is to obtain an estimate for the gradient
of the solution $u$
near the support of the atom.
Estimate (\ref{LocalPart}) follows immediately from Theorem
\ref{RHBoundary}, Lemma \ref{Energy}, and the normalization of the atom.

The next step is to estimate $ \int _ {\Sigma _k } | \nabla u  |^ p\,
d\sigma  $ for $k\geq 4$.  We begin by proving that the solution $u$
satisfies the  upper bound
\begin{equation}\label{gclaim}
|u(y)| \leq \frac{Cr^{\beta}}{|x-y|^{n-2+\beta}}, \quad |x-y|> 2r,
\end{equation}
where $\beta$ is as in Lemma \ref{Green}. To establish (\ref{gclaim}),
we use the representation formula in part
3) of Lemma \ref{Green} and claim that there exists $\bar{x}$ in $\sball x r$ such that
$$
u(y) = -\int_{\sball x r \cap N} a(z) (G(y,z)-G(y,\bar{x}))\, d\sigma.
$$
If $\sball x r \subset N$, then let $\bar{x} = x$ and use that $a$ has
mean value zero to obtain the estimate (\ref{gclaim}). If $\sball x r \cap D
\neq \emptyset$, then choose $\bar{x} \in D\cap \sball x r$ and use that $G(y,\cdot)$
vanishes on $D$. Now estimate (\ref{gclaim}) follows easily from the
normalization of the atom and the estimates for the Green function in part 4) of Lemma
\ref{Green}.

The remainder of the proof of estimate (\ref{Decay}) follows from Theorem \ref{RHBoundary} and
estimate (\ref{gclaim}). The constant in the estimate will depend on $p$, $M$, the dimension $n$,
and the collection of coordinate cylinders.
\end{proof}

Next we prove that the non-tangential maximal
function of a weak solution lies in $L^1(\partial\Omega)$ when
the Neumann data is given by an atom. We introduce the following
notation to be used in the proof.
Let
$$
\mathcal{C}_t  =  \{ \, y: t< \delta(y) < 2t \, \} \cap \Omega
\qquad\mbox{and}\qquad
C_t  =  \{\, y: t< \delta(y) < 2t\, \} \cap \partial\Omega.
$$

\begin{theorem}\label{Hardythm}
Let $\Omega$ and $D$ satisfy (\ref{Lip}) and (\ref{NTA}). Let
$\Lambda$ satisfy (\ref{SurfProp}) with
$0\leq \epsilon < (q_0-2)/(q_0-1)$ with $q_0>2$ from Lemma \ref{RHEstimate}.
If $f_N \in H^{1}(N)$, then there exists a solution
$u$ of the $L^1$-mixed problem (\ref{MP}) with Neumann data $f_N$ and zero Dirichlet
data. This solution satisfies
\begin{equation*}\label{L1est}
\|(\nabla u)^{*}\|_{L^{1}(\partial\Omega)} \leq C\|f_N\|_{H^{1}(N)}.
\end{equation*}
The constant $C$ above depends on the global character of the domain.
\end{theorem}

\begin{proof}
To begin, let $f_N=a$ be an atom for $N$ and let $u$ be the weak
solution of the mixed problem with Neumann data $f_N$ and
zero Dirichlet data. The $H^{1}(N)$ estimate will follow immediately
from the estimate for an atom.

We wish to establish a representation formula for the
gradient of $u$ in terms of the boundary values of $u$. Let
$x\in\Omega$ and let $j$ be an index ranging from $1$ to $n$.
The claim is that
\begin{eqnarray}\label{represent}
\frac{\partial u}{\partial x_j} (x) = -\int_{\partial\Omega}
\sum_{i=1}^{n} \frac{\partial \Xi}{\partial y_i} (x-\cdot)
(\nu_i \frac{\partial u}{\partial y_j} -
 \frac{\partial u}{\partial y_i}\nu_j)\nonumber \\
+ \frac{\partial \Xi}{\partial y_j} (x-\cdot)
\frac{\partial u}{\partial \nu} \, d\sigma,
\end{eqnarray}
where $\Xi$ is the fundamental solution of the
Laplacian. In the case that $u$ is smooth up to the boundary,
the formula follows from the
divergence theorem. However, it will take more work to prove
(\ref{represent}) when $u$ is only a weak solution.

Let $\eta$ be a smooth function that is zero in a neighborhood
of $\Lambda$ and supported in a coordinate cylinder. Using the
coordinates of the coordinate cylinder, let $u_{\tau}(y) = u(y+\tau e_n)$,
where $e_n$ is the  unit vector in the $n$-th direction.
Applying the divergence formula gives
\begin{eqnarray}\label{dadgrumidentity}
-\int_{\partial\Omega} \eta \left(  \sum _{i=1}^n( \frac{\partial
  \Xi}{\partial y_i} (x-\cdot) (\nu_i \frac{\partial
  u_{\tau}}{\partial y_j} -
\frac {\partial u_ \tau }{ \partial y_i}  \nu_j )) + \frac{\partial\Xi}{\partial y_j} (x-\cdot)\frac{\partial u_{\tau}}{\partial \nu}\right) \, d\sigma \nonumber \\
= \eta(x) \frac{\partial u_{\tau}}{\partial x_j}(x) - \int_{\Omega} \nabla \eta \cdot \nabla_y \Xi (x-\cdot) \frac{\partial u_{\tau}}{\partial y_j} \nonumber
 - \nabla_y \Xi(x-\cdot) \nabla u_{\tau} \frac{\partial\eta}{\partial y_j} \nonumber\\
 + \frac{\partial \Xi}{\partial y_j} (x-\cdot) \nabla u_{\tau} \cdot \nabla \eta \, dy.
\end{eqnarray}
\comment{
\begin{eqnarray}\label{dadgrumidentity}
-\int_{\partial\Omega} \eta \left( \frac{\partial \Xi}{\partial \nu} (x-\cdot) \frac{\partial u_{\tau}}{\partial y_j} -\nabla_y \Xi (x-\cdot) \nabla u_{\tau} \nu_j + \frac{\partial\Xi}{\partial y_j} (x-\cdot)\frac{\partial u_{\tau}}{\partial \nu}\right) \, d\sigma \nonumber \\
= \eta(x) \frac{\partial u_{\tau}}{\partial x_j}(x) - \int_{\Omega} \nabla \eta \cdot \nabla_y \Xi (x-\cdot) \frac{\partial u_{\tau}}{\partial y_j} \nonumber
 - \nabla_y \Xi(x-\cdot) \nabla u_{\tau} \frac{\partial\eta}{\partial y_j} \nonumber\\
 + \frac{\partial \Xi}{\partial y_j} (x-\cdot) \nabla u_{\tau} \cdot \nabla \eta \, dy.
\end{eqnarray}
}
Since the cutoff function $\eta$ vanishes near $\Lambda$, we may use
the truncated maximal function estimate in Lemma \ref{Whitney} to
let $\tau$ approach 0 from above and we can conclude that the 
identity (\ref{dadgrumidentity}) continues to  hold with $u_{\tau}$ replaced by $u$.

Our next step is to remove this restriction that $ \eta $ must vanish
on $ \Lambda$.
Towards this end, suppose that $\eta$ is of the form $\eta = \eta \phi_t$, where
$\phi_t=0$ on the set $\mathcal{C}_t$, $\phi_t = 1$ on
$\Omega\setminus \mathcal{C}_{3t}$ and $|\nabla \phi_t(x)|\leq C/t$.
By our conditions on the dimension of $\Lambda$  and Lemma
\ref{interiorint} we have that
\begin{equation}\label{lambdatest}
|\mathcal{C}_t|<C t^{2-\epsilon}.
\end{equation}
According to Lemma \ref{RHEstimate},  $ \nabla u $ lies in $ L^q ( \Omega)$
for $ q < q_0$.
Using
H\"older's inequality with $ q < q_0$,
our estimate for $\nabla \phi_t$,   and   (\ref{lambdatest}), we obtain
\begin{eqnarray}\label{gtzero}
|\int_{\Omega}  \eta \nabla \phi_t \cdot \nabla_y \Xi (x-\cdot)
\frac{\partial u}{\partial y_j} \, dy | & \leq &
\frac{C \| \eta \|_{L^{\infty}}}{t} \left( \int_{\mathcal{C}_t}
|\nabla u|^q \, dy \right)^{\frac{1}{q}} \left(\int_{\mathcal{C}_t}
|\nabla \Xi|^{q'} \, dy\right)^{\frac{1}{q'}} \nonumber \\
& \leq & C t^{-1} \sigma(\mathcal{C}_t)^{1/q'}
\left(\int_{\mathcal{C}_t} |\nabla u|^q \, dy \right)^{1/q} \nonumber \\
& \leq & C t^{(2-\epsilon)(1/q')-1} \left(\int_{\mathcal{C}_t}
|\nabla u|^q \, dy \right)^{1/q},
\end{eqnarray}
where $1/q + 1/q'=1$. The  last term will go to zero as $t$ goes to
zero,  if  $\nabla u \in L^q(\Omega)$
and $(2-\epsilon)(1/q')-1>0$ or $  \epsilon < (q-2)/(q-1)$.  Our
assumption that $  \epsilon < ( q_0-2)/(q_0-1)$ implies that we may
find a $q$ for which the right-hand side of (\ref{gtzero}) vanishes as
$ t $ tends to zero.

The remaining terms in (\ref{dadgrumidentity}) can be estimated
in a similar manner, yielding
\begin{eqnarray*}
\lim_{t\rightarrow 0^{+}} -\int_{\Omega} \nabla(\phi_t \eta)
\cdot \nabla_y \Xi(x-\cdot)\frac{\partial u}{\partial y_j} -\nabla_{y}
\Xi (x-\cdot) \cdot \nabla u \frac{\partial (\phi_t \eta)}{\partial y_j}\qquad\qquad \\
 + \frac{\partial\Xi}{\partial y_j} (x-\cdot)\nabla u \cdot \nabla (\phi_t \eta) \, dy \qquad\qquad\qquad\qquad\qquad\qquad\qquad\qquad \\
 \qquad\qquad\qquad\qquad = - \int_{\Omega} \nabla \eta
 \cdot \nabla_y \Xi (x-\cdot) \frac{\partial u}{\partial y_j}
 - \nabla_y \Xi(x-\cdot) \cdot \nabla u \frac{\partial \eta}{\partial y_j} \\
 \qquad\qquad\qquad+\frac{\partial\Xi}{\partial y_j}(x-\cdot)
 \nabla u \cdot \nabla \eta \, dy.
\end{eqnarray*}
Thus we obtain (\ref{dadgrumidentity}) with $u_{\tau}$ replaced
by $u$ and $\eta$ is not required to vanish on $ \Lambda$.   Choose a partition
of unity which consists of functions that are either supported in a
coordinate cylinder, or whose support does not intersect the boundary
of $\Omega$. As $\eta$ runs over this partition, the sum gives the
representation formula for $\nabla u$ given in (\ref{represent}).
From Theorem \ref{AtomicTheorem} we have
$\nabla u \in L^{p}(\partial\Omega)$, and by the theorem of Coifman,
McIntosh and Meyer \cite{CMM:1982} it follows that
$(\nabla u)^{*} \in L^{p}(\partial\Omega)$ and thus
$(\nabla u)^{*} \in L^{1}(\partial\Omega)$ since $\Omega$ is a bounded
domain. A few more steps will provide us with the desired estimate,
$\|(\nabla u)^{*}\|_{L^{1}(\partial\Omega)} \leq C$.

Since $(\nabla u)^{*}$ lies in $L^{p}(\partial\Omega)$, we can apply
the divergence theorem twice more and obtain the
following identities
\begin{eqnarray*}
\int_{\partial \Omega} \frac{\partial u}{\partial \nu}\, d\sigma & = & 0\\
\int_{\partial\Omega} \nu_j \frac{\partial u}{\partial y_i}
- \nu_{i}\frac{\partial u}{\partial y_j}\, d\sigma & = & 0.
\end{eqnarray*}
Using these identities and the estimates for $\nabla u$ established in
Theorem \ref{AtomicTheorem}, we can conclude that the two integrands
above are molecules on the boundary of the domain, and hence it follows
from the representation formula (\ref{represent}) that $(\nabla u)^{*}$
lies in $L^{1}(\partial\Omega)$ and satisfies the estimate
\begin{equation*}
\|(\nabla u)^{*}\|_{L^{1}(\partial\Omega)} \leq C.
\end{equation*}
The estimate for solutions with Neumann data in $H^1(N)$ follows
easily from the result for solutions with atomic data. 
\end{proof}

\section{Uniqueness}

We now turn to establishing uniqueness of solutions to the
mixed problem. We rely on the results of the previous section
and uniqueness of the regularity problem due to Dahlberg and
Kenig \cite{DK:1987} (also, see the work of D.~Mitrea \cite[Corollary
  4.2]{MR1883390} for 
the result in two dimensions).  More specifically, we prove that if $u$
solves (\ref{MP}) with zero data and $(\nabla u)^{*} \in L^{1}(\partial\Omega)$,
then $u$ also solves the regularity problem with zero data and hence $u=0$.

\begin{theorem}\label{unique}
If $u$ is a solution of the $L^1$-mixed problem (1.1)  with $ f_D=0$ and
$f_N=0$, then $u=0$. 
\end{theorem}

\noindent The proof of this theorem closely follows the
proof of uniqueness in the  paper of Ott and Brown \cite[Section 5]{OB:2009}.
We outline the main steps of the argument here and will omit many of the
technical details.

Recall the following approximation scheme of
G. Verchota (\cite[Appendix~A]{GV:1982} and \cite[Theorem~1.12]{GV:1984}). 
Given a Lipschitz domain $\Omega$, there exists a family of
smooth domains $\{\Omega_k\}$ with $\bar{\Omega}_k\subset \Omega$,
$k=1,2,\ldots$, and a family of bi-Lipschitz homeomorphisms $\Lambda_k: \partial\Omega_k \rightarrow \partial\Omega$.

The following lemma is proved via the method of Verchota. The argument
is sketched in Ott and Brown \cite[Section ~ 5]{OB:2009}. The proof uses
generalized Riesz transforms and also relies on the H\"{o}lder
continuity of the Green function established in Lemma \ref{Green}.  We
use a similar argument in Section \ref{weightedsection}.

\begin{lemma} \label{Verchota}
Let $\{\Omega_k\}$ be a family of smooth domains.
Let $u \in W^{1,1}(\partial\Omega_k)$ for $k=1,2,\ldots$. If
$w$ is a weak solution of the mixed problem in $\Omega_k$ with Neumann data
an atom for $N$ and zero Dirichlet data, then
\begin{equation*}
|\int_{\partial\Omega_k} u \frac{\partial w}{\partial \nu} \, d\sigma
| \leq C_w \| u \|_{W^{1,1}(\partial\Omega_k)}.
\end{equation*}
\end{lemma}

Next we introduce a Poincar\'{e} inequality that will be
employed below. Let $N_{\epsilon} = \{x\in N :  \delta(x)<\epsilon\}$.
We show that there is a constant $C $ such that  for $u \in W^{1,1}(\partial\Omega)$,
with $u=0$ a.e. on $D$,
\begin{equation}\label{Hardy1}
\frac{1}{\epsilon} \int_{N_{\epsilon}} |u| \, d\sigma
\leq C \int_{N_{C\epsilon}} |\nabla u |\, d\sigma.
\end{equation}
To prove the inequality above, let $x\in N_{\epsilon}$ and let
$Q_{x,\epsilon}$ denote the surface cube centered at $x$ with side
length $\epsilon$.  If $ x \in N_ \epsilon $, then $ Q_ { x, 2 \epsilon
} \cap D \neq \emptyset$, thus we may use (\ref{Dbig}) to obtain the
Poincar\'{e} inequality
\begin{equation*}
\int_{Q_{x,4\epsilon}} |u|\, d\sigma \leq  C\epsilon
\int_{Q_{x,4\epsilon}} |\nabla u| \, d\sigma.
\end{equation*}
\note{We needed to expand the cubes a little to guarantee that the
  Poincar\'e inequality holds. }
By the Besicovitch covering lemma, we can find a finite
cover of $N_\epsilon$ by cubes $Q_{x_1, 4\epsilon}, Q_{x_2,4\epsilon},
\ldots, Q_{x_m,4\epsilon}$ such that the cubes have bounded overlaps,
\begin{equation*}
\sum_{i=1}^m  \chi_{Q_{x_i},4\epsilon}\leq C_n.
\end{equation*}
Then we have
\begin{eqnarray*}
\int_{N_\epsilon} |u|\, d\sigma & \leq &
\sum_{i=1}^m  \int_{Q_{x_i,4\epsilon}} |u|\, d\sigma \\
& \leq & C\sum_{i=1} ^m  \epsilon \int_{Q_{x_i,4\epsilon}}|\nabla u|\, d\sigma \\
& \leq & C \epsilon \int_{N_{C\epsilon}} |\nabla u|\, d\sigma.
\end{eqnarray*}

Thus we obtain (\ref{Hardy1}) which we will use to prove the
following approximation lemma. Let $\Upsilon$ denote the
collection of functions defined in $\bar{\Omega}$ that are Lipschitz and compactly supported in
$ \partial \Omega \setminus \bar D$.
The next lemma shows that a function
$u \in W^{1,1}(\partial\Omega)$ which satisfies $u=0$ a.e. on
$D$ can be  approximated in the $\sobolev 1 1 ( \partial \Omega)
$-norm by functions in $ \Upsilon$.
 This density is fairly easy for
the domains considered in
Ott and Brown \cite{OB:2009}, but requires more work under the present
assumptions.

\begin{lemma}\label{approx1}
Let $\Omega$ and $D$ satisfy (\ref{Lip}) and (\ref{Dbig}).
Suppose that $u \in W^{1,1}(\partial\Omega)$ and $u=0$ a.e. on $D$.
Then $u$ can be approximated in $ \sobolev 1 1 ( \partial \Omega)$ by
functions from $ \Upsilon $.
\end{lemma}

\begin{proof}
Let $u\in W^{1,1}(\partial\Omega)$ and suppose that $u=0$ a.e. on $D$.
Fix $\epsilon > 0$ and let
$\eta_{\epsilon}$ be a smooth function which is 1 if $t> 2\epsilon$
and $0$ if $t< \epsilon$, and satisfies $|\nabla \eta_{\epsilon}(x)|< C/\epsilon$.
If $u \in W^{1,1}(\partial\Omega)$ and
vanishes a.e.~on $D$, then we have that
$u_{\epsilon}(x) = \eta_{\epsilon}(\delta(x))u(x)$ is zero in  a
neighborhood of $ D\cup \Lambda $.
 Then it follows that
$\nabla_t u_\epsilon - \nabla_t u = u \nabla_t
\eta_{\epsilon}(\delta(\cdot)) + (\eta_{\epsilon}(\delta(\cdot))-1)\nabla_t u$.
From the dominated convergence theorem,
\begin{equation*}
\lim_{\epsilon\rightarrow 0^+} \int_N |\nabla_t u|
|\eta_{\epsilon}(\delta(\cdot))-1| \, d\sigma = 0.
\end{equation*}
Since $\nabla \eta_{\epsilon}(\delta(\cdot))\leq C/\epsilon$,
we may use the Poincar\'{e} inequality (\ref{Hardy1}), which
requires (\ref{Dbig}), and
the dominated convergence theorem to conclude that
\begin{equation*}
\lim_{\epsilon\rightarrow 0^+} \int_N |u \nabla_t
\eta_{\epsilon}(\delta(\cdot))| \, d\sigma = 0.
\end{equation*}
Thus, we have that $u\in W^{1,1}(\partial\Omega)$ may be
approximated by a function $u_{\epsilon}$ that is supported in $
\partial \Omega \setminus \bar D$. By a standard
regularization argument,  we can approximate $u_{\epsilon}$
by functions that are in $\Upsilon$.
\end{proof}

\begin{proof}[Proof of Theorem \ref{unique}]
Let $u$ be a solution of the mixed problem (\ref{MP}) with
$f_N=0$ and $f_D=0$. We wish to show that $u=0$. Fix an atom
$a$ for $N$ and let $w$ be a solution of the mixed problem with
$f_N=a$ and $f_D=0$ as constructed in Theorem \ref{AtomicTheorem}. Our
goal is to show that
\begin{equation}\label{AtomClaim}
\int_N ua \, d\sigma = 0.
\end{equation}
In turn, this will imply that $u$ is zero on $\partial\Omega$, and by
appealing to the uniqueness of the regularity problem  proved by
Dahlberg and Kenig \cite{DK:1987} or D.~Mitrea \cite{MR1883390} in two
dimensions, we can conclude that $u=0$ in $\Omega$.

To prove (\ref{AtomClaim}), we apply Green's second identity
in one of the smooth approximating domains from
Verchota's construction and obtain
\begin{equation}\label{uniq42}
\int_{\partial\Omega_k} w \frac{\partial u}{\partial \nu} \, d\sigma
= \int_{\partial\Omega_k} u \frac{\partial w}{\partial \nu}\, d\sigma,
\quad k=1,2,\ldots.
\end{equation}
We have that $(\nabla u)^* \in L^{1}(\partial\Omega)$ and Lemma
\ref{Green} implies that  $w$ is
H\"{o}lder continuous and hence bounded. Further, $w=0$ on $D$ and
$\frac{\partial u}{\partial \nu}=0$ on $N$. Hence by the dominated
convergence theorem,
\begin{equation}\label{uniq43}
\lim_{k\rightarrow\infty} \int_{\partial\Omega_k} w
\frac{\partial u}{\partial \nu}\, d\sigma =0.
\end{equation}
Thus, we can prove our claim by showing that
\begin{equation}\label{ClaimFollow}
\lim_{k\rightarrow\infty} \int_{\partial\Omega_k}
u \frac{\partial w}{\partial \nu}\, d\sigma = \int_{\partial\Omega}
ua \, d\sigma.
\end{equation}
Note that the existence of the limit in (\ref{ClaimFollow})
follows from (\ref{uniq42}) and (\ref{uniq43}).
Now by repeating the  argument
used to prove Lemma 5.7 in the work of
Ott and Brown \cite{OB:2009},
we can find a sequence $\{U_j\}$  of Lipschitz functions defined in $ \bar
\Omega$ such that  $ U_j |_D=0$ and
\begin{equation}\label{uj} 
\lim_{k\rightarrow \infty} \|u-U_j\|_{W^{1,1}(\partial\Omega_k)} \leq 1/j.
\end{equation}
The argument outlined above uses the density result in Lemma \ref{approx1}.
Now we have
\begin{eqnarray}
|\int_{\partial\Omega} ua \, d\sigma - \lim_{k\rightarrow\infty} \int_{\partial\Omega_k} u\frac{\partial w}{\partial \nu} \, d\sigma| \leq |\int_{\partial\Omega} ua \, d\sigma - \lim_{k\rightarrow \infty} \int_{\partial\Omega_k} U_j \frac{\partial w}{\partial \nu}\, d\sigma | \nonumber \\
\quad + \lim \sup_{k\rightarrow \infty} |\int_{\partial\Omega_k} (u-U_j) \frac{\partial w}{\partial \nu}\, d\sigma|.\label{U2}
\end{eqnarray}
Since $(\nabla w)^* \in L^1 (\partial\Omega)$ and $U_j$
is bounded, we may take the limit of the first term on the
right of (\ref{U2}). This yields
\begin{equation*}
|\int_{\partial\Omega} ua \, d\sigma - \lim_{k\rightarrow \infty}
\int_{\partial\Omega_k} u\frac{\partial w}{\partial \nu} \, d\sigma|
\leq | \int_{\partial\Omega} (u-U_j)a \, d\sigma| + C/j \leq C/j.
\end{equation*}
The second term on the right of (\ref{U2}) is bounded by $C_w/j$
by Lemma \ref{Verchota} and (\ref{uj}).  Since $j$ is arbitrary,
we have obtained (\ref{ClaimFollow}) and the proof of the Theorem
is complete.
\end{proof}

\section{$L^p$ result} \label{lpsection}

In this section, we use the existence of solutions of the
mixed problem with data from Hardy spaces
established in Section \ref{Sect3} to prove
$L^p$-estimates for the mixed problem. Our strategy is to
first recall the reverse H\"{o}lder inequality  at the boundary
which was proved in Theorem \ref{RHBoundary}.
With this estimate in hand, we then apply the method developed
by Shen \cite{ZS:2007} and adapted by Ott and Brown \cite{OB:2009} to obtain
the $L^p$-estimate.

The following lemma is a local estimate that is a consequence of
Theorem \ref{RHBoundary}.  In this Lemma we use the truncated
non-tangential maximal function defined in Section \ref{prelim}.

\begin{lemma} \label{newLocal}
Let $\Omega$ and $D$ satisfy assumptions (\ref{Lip}) and (\ref{NTA}).
Let $q_0>2$ be
as in Lemma \ref{RHEstimate} and let $\Lambda$ satisfy (\ref{SurfProp}) with
$0\leq \epsilon < (q_0-2)/(q_0-1)$.
Let $u$ be the weak solution of the mixed problem with $f_N\in L^{2}(N)$ and zero
Dirichlet data.
Let $x\in \Omega$ and $0<r<r_0$. Then for
$1<p<q_0((1-\epsilon)/(2-\epsilon))$ the following
local estimate holds
\begin{equation*}
\left(\average_{\Delta_{r} (x)}(\nabla u)^{*p}_{cr} \, d\sigma \right)^{1/p} \leq C \left( \average_{\Psi_{2r}(x)} |\nabla u|\, dy +
r ^ { (1-n)/2}\|f_N\|_{L^2(\sball x {2r} \cap N)} \right).
\end{equation*}
The constant $c=1/16$ and $C$ depends on $M$ and $n$.
\end{lemma}

\begin{proof}
Let $x\in\Omega$ and $r$ satisfy $0<r<r_0$.
Theorem \ref{RHBoundary} provides  an estimate
for the $L^p$-norm of $\nabla u$ in a surface ball $\Delta_r (x)$.
To obtain the estimate for the non-tangential maximal function,
choose a cut-off function $\eta$ which is one
on $B_{3r}(x)$ and supported in $B_{4r}(x)$. Let $z\in B_r (x)$.
By repeating the argument used to prove (\ref{represent})  
in the proof of Theorem \ref{Hardythm}, we can show that the gradient
of the weak solution $u$ may be represented as
\begin{eqnarray*}
(\eta\frac {\partial  u}{\partial z_j}) (z)
 &  = &  \int _ { \partial
    \Omega }
\eta ( \frac {\partial \funsol }{ \partial \nu } ( z-\cdot ) \frac {
  \partial u }{\partial y _j }
- \nu _j  \nabla_y \funsol (z- \cdot ) \cdot \nabla u
+ \frac { \partial \funsol }{\partial y _j }(z-\cdot) \frac  {
  \partial u }{ \partial \nu }  ) \, d\sigma \\
& &
\qquad - \int _ { \Omega} \nabla \eta \cdot \nabla_y \funsol ( z-\cdot )\frac {
  \partial u }{ \partial y _j }
- \frac { \partial \eta }{ \partial y _j } \nabla_y \funsol (
z-\cdot)\cdot \nabla u  \\
& & \qquad \qquad  \qquad + \nabla \eta \cdot \nabla u \frac { \partial
  \funsol }{ \partial y _j } ( z- \cdot ) \, dy.
\end{eqnarray*}
From this representation and the theorem of Coifman, McIntosh and Meyer \cite{CMM:1982}
on the boundedness of the Cauchy integral, we get
\begin{equation}\label{eq62}
\left(\average_{\sball x r} (\nabla u)_{r}^{*p}\, d\sigma\right)^{1/p}
\leq C \left[\average_{\dball x {4r}} |\nabla u|\, dy + \left(
\average_{\sball x {4r}} |\nabla u|^p \, d\sigma\right)^{1/p}\right].
\end{equation}
From estimate (\ref{eq62}), Theorem \ref{RHBoundary}, and a covering argument, we obtain
the Theorem.
\end{proof}

Next, we outline the argument developed by Shen \cite{ZS:2007}
that we employ to obtain $L^p$-estimates in this section and weighted
$L^p$-estimates in the next section. Shen's argument is adapted
from work of Peral and Caffarelli \cite{MR1486629}.
It depends on a Calder\'{o}n-Zygmund decomposition of the boundary
and thus we will use surface cubes in this section rather than
surface balls $\Delta_r (x)$.
Before giving Shen's result,
recall that
a locally integrable
function $w$ is an \emph{$A_p(d\sigma)$ weight}, $1<p<\infty$,
provided that
\begin{equation}\label{apw}
\frac{1}{\sigma(\Delta)} \int_{\Delta} w \, d\sigma
\left(\frac{1}{\sigma(\Delta)} \int_{\Delta} w^{-p'/p}\,d\sigma\right)^{p/p'} \leq A < \infty,
\end{equation}
for all surface balls $\Delta\subset\partial\Omega$ centered on $\partial\Omega$.
Define
$A_{\infty}(d\sigma) = \cup_p A_p (d\sigma)$.

Let $Q_0$ be a surface cube and let $F$ be defined on $4Q_0$. Let
the exponents $p, q$ satisfy $1<p<q$. Assume that for each
$Q\subset Q_0$, we may find functions $F_Q$ and $R_Q$, defined
in $2Q$, satisfying
\begin{equation}\label{Shen1}
|F|\leq |F_Q| + |R_Q|,
\end{equation}
\begin{equation}\label{Shen2}
\average_{2Q} |F_Q| \, d\sigma \leq C\left( \average_{4Q} |f|^p\, d\sigma\right)^{1/p},
\end{equation}
\begin{equation}\label{Shen3}
\left(\average_{2Q} |R_Q|^q \, d\sigma \right)^{1/q} \leq C \left[ \average_{4Q} |F|\,d\sigma + \left(\average_{4Q} |f|^p \, d\sigma\right)^{1/p}\right].
\end{equation}
Going further, assume that $\mu $ is a weight in $  A_t(d\sigma)$ and that
\begin{equation}
\label{Weight4}
\left(\frac{\mu(E)}{\mu(Q)} \right)\leq
C\left(\frac{\sigma(E)}{\sigma(Q)}\right)^\theta, \quad 1<t<\theta q,
\end{equation}
\comment{
Then
\begin{equation}\label{star2}
\left(\average_{Q_0} |F|^{t}\, \mu d\sigma \right)^{1/t}
\leq C \left[ \average_{4Q_0} |F| \, d\sigma +
\left(\average_{4Q_0} f^t \, \mu d\sigma \right)^{1/t}\right].
\end{equation}
}
where $\theta$ depends on $M$. Under the assumptions (\ref{Shen1})--(\ref{Weight4}), for $s$ in the
interval $(p,\theta q)$, we have 
\begin{equation}\label{Shenconc}
\left(\average_{Q_0} |F|^{s}\, \mu d\sigma \right)^{1/s}
\leq C \left[ \average_{4Q_0} |F| \, d\sigma +
\left(\average_{4Q_0} |f|^s \, \mu d\sigma \right)^{1/s}\right].
\end{equation}
\comment{
\begin{equation}
\left(\average_{Q_0} |F|^s \, d\sigma \right)^{1/s}
\leq C \left[\average_{4Q_0} |F|\, d\sigma +
\left(\average_{4Q_0} |f|^s\, d\sigma \right)^{1/s}\right].
\end{equation}
}
The constant in the estimate above depends on the
Lipschitz constant of $\Omega$ and the constants in the
estimates (\ref{Shen2})--(\ref{Weight4}).  The argument to
obtain (\ref{Shenconc}) is essentially the same as in
Shen \cite[Theorems 3.2, 3.4]{ZS:2007}. Ott and Brown \cite[Section 7]{OB:2009}
rework Shen's argument to apply to the current situation, where
our starting point is a result in a Hardy space rather than
in an $L^p$-space.

Let $4Q_0$ be a surface cube with sidelength comparable to
$r_0$. Let $u$ be a solution of the mixed problem with Neumann
data $f\in L^{p}(N)$ and zero Dirichlet data. Since
$L^{p}(N)\subset H^{1}(N)$, we know by Theorem \ref{AtomicTheorem}
that a solution $u$ exists with
$(\nabla u)^{*}\in L^{1}(\partial\Omega)$. Let $F=(\nabla u)^{*}$.
Now given a cube $Q\subset Q_0$ with diameter $r$, we define
$F_Q$ and $R_Q$ as follows. Let $\bar{f}_{4Q}=0$ if
$4Q\cap D\neq \emptyset$ and $\bar{f}_{4Q}=\average_{4Q}f \, d\sigma$
if $4Q\subset N$. Set $g=\chi_{4Q}(f-\bar{f}_{4Q})$ and $h=f-g$.
By construction,  $g$ and $h$ are both elements of $H^{1}(N)$
and thus we may solve the mixed problem with Dirichlet data
zero and Neumann data $g$ or $h$. Let $v$ solve the  mixed
problem with Neumann data equal to $g$ and let $w$ solve the
mixed problem with Neumann data $h$. By our uniqueness result
Theorem \ref{unique}, we have that $u=v+w$. Let
$F_Q=(\nabla v)^{*}$ and $R_Q=(\nabla w)^{*}$. It follows
immediately that (\ref{Shen1}) holds. The procedures for obtaining
estimates (\ref{Shen2}) and (\ref{Shen3}) are straightforward
and were worked out in detail in Ott and Brown \cite{OB:2009}.

With (\ref{Shen1})--(\ref{Shen3}) established, we obtain
(\ref{Shenconc}) for $s$ in the interval $(p,\theta q)$ and
$q<q_0((1-\epsilon)/(2-\epsilon))$. From here, we can
now easily complete the $L^p$-estimate.

\begin{proof}[Proof of Theorem \ref{main}]
To prove part b), we use Dahlberg and Kenig's result \cite[Theorem 4.3]{DK:1987}
for the regularity problem  in Hardy spaces,
or D. Mitrea's result \cite{MR1883390} in two dimensions,
 to reduce to
 the case where
the Dirichlet data is zero. Then we take $f_N \in H^{1}(N)$
and use Theorem \ref{Hardythm} to complete the proof.

To prove part a), consider the mixed problem with zero
Dirichlet data and Neumann data in $L^{p}(N)$. Since
$L^{p}(N)$ is contained in the Hardy space $H^{1}(N)$, a
solution $u$ exists by part b). From the argument of Caffarelli
and Peral, as adapted by Shen, and Ott and Brown, we have that
$u$ satisfies the estimate
\begin{equation*}
\|(\nabla u)^{*} \|_{L^{p}(\partial\Omega)} \leq C\| f_N \|_{L^{p}(N)}.
\end{equation*}

Uniqueness of solutions of the mixed problem follows from Theorem \ref{unique}.
\end{proof}

\section{Weighted result} \label{weightedsection}

In this section we establish results for the mixed
problem with data from weighted Sobolev spaces. Throughout
this section we assume that $\Omega$ and $D$
satisfy conditions (\ref{Lip}), (\ref{SurfProp}),
and (\ref{NTA}).

To begin, we consider the regularity problem when
the data comes from a weighted Sobolev space.  We will use the
solution of the regularity problem to reduce the study of the mixed
problem to the case when the Dirichlet data is zero.
Our study of the regularity problem contained here is a small extension of work of
Shen \cite{ZS:2005} who studied the regularity problem with data
in weighted $L^2$-Sobolev spaces.
Shen's work is in turn an extension of a method used by
Verchota \cite{GV:1984} to study the (unweighted) regularity problem
in Lipschitz domains.  This method is also developed in
a recent article of Kilty and Shen \cite{KS:2010} that  studies
the relationship between the regularity problem and the Dirichlet
problem for elliptic systems.   We choose to repeat well-known arguments for several
reasons. Kilty and Shen do not give  weighted estimates and there is
a small mistake in \cite{ZS:2005}. The weight defined in equation (7.29)
on page 2868 of \cite{ZS:2005} may not be a doubling weight and hence
may  not be in any $ A_p $ class.

The heart of the matter is Lemma \ref{lma} below, which estimates
the normal derivative of a harmonic function in terms of its
boundary values. Building toward this result, we begin by recalling
that Verchota's result for the regularity problem with data in
unweighted Sobolev spaces depends on a duality argument and the solution of
the Dirichlet problem with data in weighted $L^p$-spaces.  Thus, our starting point will
be the following result of Dahlberg \cite{BD:1977} regarding the weighted
Dirichlet problem.  

In the results that follow, constants have the dependencies given in
Section 2. In addition, the  constant may depend on the weight through
the $A_p$-constant and the exponents appearing in the assumptions on the
weights. 

We begin by recalling some well-known results about the  $L^p(\mu \,
d\sigma)$-Dirichlet problem. In this problem, given $f$ on the
boundary we  look for a
function $u$ which satisfies
\begin{equation} \label{DIR1}
\left \{ \begin{array}{ll}
\Delta u = 0, \qquad & \mbox{in } \Omega \\
u = f , \qquad & \mbox{on } \partial \Omega \\
u^* \in L^ p ( \mu \, d\sigma) . 
\end{array} \right. 
\end{equation}

\begin{theorem}
\label{Dahlberg}
There exists an exponent $ s_0 <2$ such that if
 $\mu \in A_r (d\sigma)$, $r>1$,   $p>rs_0$,  and $f \in L^ p ( \mu \,
d\sigma)$,   then the $L^p( \mu \, d\sigma)$-Dirichlet
has a  unique  solution. 
\end{theorem}

\begin{proof}
Dahlberg  \cite{BD:1977}  has shown that
there exists an exponent $ t_0 > 2$ such that the
harmonic measure lies in the reverse H\"older class $ \rh _t (
d\sigma)$ for $ t< t_0$, meaning that for each $ t < t_0$, there is a
constant $C_t$ such that
\begin{equation} \label{rhomega}
\left(\average_\Delta \omega^t \,d\sigma \right)^{1/t} \leq
C_t \average_{\Delta} \omega \, d\sigma,
\end{equation}
for any surface ball $\Delta$ centered on $\partial\Omega$.
Here, $ \omega$ denotes the density with respect to surface measure  of harmonic measure at some
convenient point in $ \Omega$.
  The exponent $ s_0$ will be the dual exponent to $ t _0$, \textit{i.e.}, $1/s_0 + 1/t_0=1$.

Let $f\in L^{2}(d\sigma) \cap L^p(\mu\, d\sigma)$, and let $u$ be the solution
of the $L^2 ( d\sigma)$-Dirichlet problem. 
From Hunt and Wheeden \cite{MR0274787} (see also Jerison and Kenig
\cite{JK:1982a}),
we know that
\begin{equation*}
u^* (x) \leq C M_{\omega}f(x),
\end{equation*}
where $M_{\omega} f$ is the Hardy-Littlewood maximal function
with respect to harmonic measure given by
\begin{equation*}
M_{\omega} f(x) = \sup_{r>0} \frac{1}{\omega(\Delta_r (x))} \int_{\Delta_r (x)}
|f| \, \omega \,  d\sigma.
\end{equation*}
Since $\omega\in \rh_t (d\sigma)$ for $t<t_0$, we have
\begin{equation*}
M_{\omega} f(x) \leq C_s M(|f|^s)^{1/s} (x), \quad s_0<s<\infty,\  1/s_0 + 1/t_0=1.
\end{equation*}

Next, set $s=p/r$ and note that our assumption $p>rs_0$ implies that
$s>s_0$. Since $\mu \in A_r (d\sigma)$, it follows that with this
choice of $s$ we have the estimate
$$
\| M(|f|^s)^{1/s}\|_{L^p (\mu \,   d\sigma)}
\leq C \|f\|_{L^p(\mu  \,   d\sigma)},
$$
which implies that
\begin{equation}\label{star}
\|u^*\|_{L^{p}(\mu \,  d\sigma)} \leq C\|f\|_{L^p(\mu \,  d\sigma)}.
\end{equation}
By a standard limiting argument, we may therefore construct solutions
 $u\in L^{p}(\mu \,  d\sigma)$  to the
Dirichlet problem (\ref{DIR1}) which satisfy
the estimate (\ref{star}).

Finally, to establish uniqueness of solutions of (\ref{DIR1}), 
observe that $L^p (\mu\, d\sigma) \subset
L^s (d\sigma)$ when $s=p/r$. Since $ s> s_0$, we  may use the
uniqueness result for the 
$L^s (d\sigma)$-Dirichlet problem to conclude that if
$u$ is harmonic in $ \Omega$,
$u^*\in L^{p}(\mu \,  d\sigma)$, and $u$ has non-tangential limits of 0 a.e. on
$\partial\Omega$, then $u=0$.
\end{proof}

\comment{
\begin{proposition}\label{prop1}
Let $\omega$ be the density of harmonic measure with
respect to surface measure $d\sigma$. There is an exponent $t_0>2$
such that $\omega$ lies in the reverse H\"older class
$\rh_t (d\sigma)$ for $t<t_0$. In other words,
for all $t<t_0$ there exists a constant $C=C_t$ such that
\begin{equation*}
\left(\average_\Delta \omega^t \,d\sigma \right)^{1/t} \leq
C\average_{\Delta} \omega \, d\sigma
\end{equation*}
for all surface balls $\Delta\subset \partial\Omega$.
\end{proposition}
}

The next theorem establishes solvability of the regularity problem
when the boundary data lies in a weighted Sobolev space. 
Given a function $f$ on the boundary, the $ L^ { p'} ( \mu ^ {
  -p'/p}\, d\sigma)$-regularity problem is the problem of finding a 
function $u$ which satisfies
\begin{equation*}
\left\{ \begin{array}{ll}
\Delta u = 0, \quad & \mbox{in }\Omega\\
u = f, \quad & \mbox{on }\partial\Omega\\
(\nabla u)^* \in L^p(\mu^{-p'/p}\, d\sigma).
\end{array}\right.
\end{equation*}

\begin{theorem}\label{thmb}
Let $s_0$ be as in Theorem  \ref{Dahlberg} and
 let $ \mu \in A_ r( d\sigma)$, $r>1$.
If  $\infty>  p > r s_0$  and $ f$ lies in  $  W^{1,p'}( \mu
^{-p'/p}d\sigma)$,  then there exists a unique solution of  the
$L^{p'}( \mu ^{-p'/p} \,d\sigma)$-regularity 
problem with data $f$ which satisfies
$$
\int _ { \partial \Omega} ( \nontan{(\nabla u)}) ^ {p'} \, \mu ^ {
  -p'/p}\, d\sigma
\leq C \int _ { \partial \Omega}
( |\tangrad u |^ {p'} + |u|^{p'}  )  \,
\mu ^ {  -p'/p} \, d\sigma .
$$
\end{theorem}

In the following statement, let $\nabla_t$ denote the tangential
gradient at the boundary (see Section 2 for the definition).

\begin{lemma}\label{lma}
Let $\mu \in A_r(d\sigma)$ with $r>1$ and suppose that
$\infty > p>rs_0$, where $s_0$ is as in Theorem  \ref{Dahlberg}.
If $u$ is a harmonic function with
$(\nabla u)^* \in L^2(d\sigma )$, then
\begin{equation*}
\| \frac{\partial u}{\partial \nu} \|_{L^{p'}(\mu^{-p'/p} \,  d\sigma)} \leq
C \left( \|\nabla_t u\|_{L^{p'}(\mu^{-p'/p}\,  d\sigma)} + \|u\|_{L^{p'}
(\mu^{-p'/p} \,  d\sigma)}\right).
\end{equation*}
\end{lemma}

To prove this lemma, we begin by defining {\em local Riesz
  transforms}.
Fix a coordinate cylinder $Z_{r}(x)$, $ r < r_0$,  for $\partial\Omega$ such that
$Z_{2r}(x)$ is also a coordinate cylinder. Hereafter in this section
we will use $Z_r$ to denote $Z_r(x)$. Let $\eta$ be a smooth cutoff
function such that $\eta = 1$ on $Z_{r}$ and $\eta = 0 $ outside $Z_{3r/2}$.
Let $v$ be a harmonic function. Using coordinates $(x', x_n ) \in
\reals ^ { n-1} \times \reals$, for  $ i = 1,\dots, n$,  the local Riesz transforms are given by
\begin{equation*}
v_i (x) =  -\displaystyle \int_{x_n}^{\infty} \frac{\partial}{\partial
  x_i} (\eta v) (x',t) dt,
\end{equation*}
for $x\in Z_{2r}\cap \Omega$, and $v_i (x) = 0$ in
$\Omega\setminus Z_{2r}$, $i = 1, \ldots, n$.
Straightforward calculations give that
\begin{eqnarray}
\label{Mixed}
\displaystyle \frac{\partial  v_i}{\partial x_j}  & = &
\frac{\partial v_j }{\partial x_i}, \\
\displaystyle \sum_{i=1}^{n}
\frac{\partial v_i}{\partial x_i}(x) & = & - \int _ { x_n} ^ \infty
\nabla \eta (x',t) \cdot \nabla v(x',t) + v(x',t) \Delta \eta ( x',t)
\, dt  \label{Divergence} \\
\Delta v_i(x) & = & - \int_ { x_n } ^ \infty
\frac \partial { \partial x_ i } ( \nabla \eta (x',t) \cdot \nabla
v(x',t) + v(x',t) \Delta \eta ( x',t) )
\, dt.
\nonumber
\end{eqnarray}

\comment{
\begin{lemma}
\label{RieszProps}
Let $w_i$ be as in (\ref{wid}). Then we have
\begin{eqnarray*}
\displaystyle \frac{\partial}{\partial x_j} w_i & = &
\frac{\partial}{\partial x_i} w_j,\\[4pt]
\displaystyle \sum_{i=1}^{n}
\frac{\partial w_i}{\partial x_i} & = & F,
\end{eqnarray*}
where $F$ is bounded in $Z_r$ and
\begin{equation*}
\sup\{ |F(y)|: y\in Z_r\} \leq \sup_K |w|,
\end{equation*}
and, $K$ is the compact subset of $ \Omega$  $K= \{(y',y_n):
|y'-x'|\leq r, \, (1+M)2r < y_n - x_n \leq (1+M)4r\}$.
\end{lemma}
}

\begin{lemma}
\label{RieszEstimate}
Fix $x\in \partial\Omega$ and $0<r<r_0$.
Let $v_1, \dots, v_n$ be the local Riesz transforms of a harmonic function $v$
in a coordinate cylinder $ Z_ { 32r } $ and suppose that $ \mu $ lies in $
A _ \infty ( d\sigma)$. Then for $p < \infty$,  the following estimate holds
$$
\int _ { \Delta _ r (x) } (\nontan { v_{i,r}} )^ p \, \mu\,  d\sigma
\leq C ( \int _ { \sball x { 8r} } (\nontan {v _ {8r}}) ^ p \, \mu\,  d\sigma
+ \mu ( \sball x { 8r}) \sup _K |v|^ p ).
$$
Above, $K\subset Z_ {32r} $ is a compact subset of $\Omega$ and the
cone opening for the non-tangential maximal function on the right is
larger than the cone opening for the non-tangential maximal function on
the left.
\end{lemma}

\begin{proof}
The proof uses a truncated square function which we define by
$$
S_r(v)(x) = \left( \int _ { \Gamma_r(x) }| \nabla v(y) |^2
|x-y|^ { 2-n}\, dy\right)^ { 1/2},
$$
where $ \Gamma _r (x)$ is a truncated cone as defined in Section \ref{prelim}.
Let $ v $ be a given harmonic function and then let $ v_ i $ be
one of the local Riesz transforms of $v$ defined in a coordinate cylinder $
Z _ { 32r }  $. Write $ v_ i = v_ i ' + v_ i ''$ where $
v_ i ' $ is harmonic in $\Omega \cap  Z _ { 16r}  $ and $ v_ i ''= \funsol * (
\Delta v_i \chi _ { Z_ { 16r} \cap \Omega})$. We observe that $
\Delta v_i$ is bounded in $ Z _ {16r}$  and
$$
\sup_ {\Omega \cap  Z _ {16r}} |\Delta v_i | \leq \frac C {r ^ 2}  \sup
_K  |v|,
$$
where $K$ is the compact set
$$
K =
 \{(y',y_n): |y'-x'|\leq 32r, \, (1+M)r \leq y_n - x_n \leq
 (1+M)32r\}.
$$
 Thus, we have
\begin{equation} \label{PrimeEstimate}
\sup _{ Z _ {16r} } | v_i '' | + r |\nabla v_i ''| \leq C
\sup
_K  |v|.
\end{equation}
With these preliminaries, we can now give the main estimate
\begin{eqnarray*}
\int _ {\sball x r } ( \nontan{v_{i,r}} )^p \, \mu\,  d\sigma & \leq &
C(\int _ {\sball x r } ( \nontan{{v'_{i,r}}})^p \, \mu \, d\sigma + \mu (
\sball x r )  \sup _K |v|^p      )
\\
 & \leq & C(  \int _ {\sball x {2r}} S_{2r} (v_i  ') ^p \, \mu \,  d\sigma + \mu (
 \sball x r )  \sup _K |v|^p      )
 \\
& \leq & C(\int _ {\sball x {4r}} S_{ 4r} (v_n  ') ^p \, \mu \,  d\sigma + \mu (
\sball x r )  \sup _K |v|^p      )  \\
& \leq &  C( \int _ {\sball x {8r} } ( \nontan{{v_{n,8r}}})^p \, \mu\,
d\sigma + \mu (
\sball x r )  \sup _K |v| ^ p  ) .
\end{eqnarray*}
The first inequality follows from (\ref{PrimeEstimate}), the second
is a local version of a theorem of Dahlberg \cite[Theorem 1]{BD:1980},
the third inequality follows from a pointwise estimate which may be
found in Stein \cite[p.~213--214]{ES:1970}, and finally the fourth inequality  follows
from Dahlberg's result and
(\ref{PrimeEstimate}). Note that in each of the  inequalities above,
the cone opening for the object of the left side must be smaller than
the cone opening for the object on the right side. Our notation is
already elaborate and thus, we choose to suppress this dependence.
Once we recall that $ v_n = \eta v$, then the estimate of the Lemma
follows.
\end{proof}

We now are ready to present a proof of Lemma \ref{lma}.

\begin{proof}[Proof of Lemma \ref{lma}]
Let $u$ be a solution of the $L^2(d\sigma)$-regularity problem with data
$f\in W^{1,2}(d\sigma)$. 
Since we may solve
the $L^2(d\sigma)$-regularity problem \cite{JK:1982a}, we may assume that $f$
is supported in a surface ball $ \sball x r$, $x\in\partial\Omega$, and that
 $ Z_ {32r}  = Z_{32r}(x) $ is  a coordinate cylinder.
Suppose that $ \partial \Omega $ is given as the graph of $ \phi$
in $ Z_{ 32r}$.
We would like to show that
\begin{equation*}
\|\frac{\partial u}{\partial \nu}\|_{L^{p'}(\mu^{-p'/p}\, d\sigma)}
\leq C \|u \|_{W^{1,p'}(\mu^{-p'/p}\, d\sigma)}.
\end{equation*}

Toward this end, choose $g\in W^{1,2}(d\sigma)$ and let $v$
be the solution of the Dirichlet problem with data $g$. We
observe that
\begin{equation*}
\|\frac{\partial u}{\partial \nu}\|_{L^{p'}(\mu^{-p'/p}d\sigma)}
= \sup_{\|g\|_{L^p(\mu\,  d\sigma)} \leq 1}
\int_{\partial\Omega} g\frac{\partial u}{\partial \nu}\, d\sigma.
\end{equation*}
Since $v=g$ on $\partial\Omega$, $u$ and $v$ are harmonic, and
$(\nabla u)^* + (\nabla v)^* \in L^2(d \sigma)$, we may use
Green's second identity to conclude that
\begin{equation*}
\int_{\partial\Omega} g \frac{\partial u}{\partial \nu}\, d\sigma
= \int_{\partial\Omega} u \frac{\partial v}{\partial \nu}\, d\sigma.
\end{equation*}
As $u=0$ on $\partial\Omega \setminus Z_r$, we may use that
$v_n = \eta v=v$ on $Z_r$, (\ref{Mixed}),  (\ref{Divergence}),  some
algebra, and integration by parts to obtain
\begin{eqnarray*}
\int_{\partial\Omega} u \frac{\partial v}{\partial \nu}\, d\sigma
& = & \int_{\partial\Omega} u \frac{\partial v_n}{\partial \nu} \, d\sigma \\
& = & \int_{\reals^{n-1}}u(x', \phi(x')) \left(-\frac{\partial
  v_n}{\partial x_n}(x',\phi(x')) \right. \\
& & \left. \qquad
+\sum_{i=1}^{n-1} \phi_{x_i} \frac{\partial v_n}{\partial x_i}
(x',\phi(x')) \right)\, dx'\\
& = & \int _ { \reals ^ { n-1} } u(x', \phi(x')) (F(x', \phi (x'))+ \sum _ { i=1} ^ { n-1}
\frac \partial { \partial x_ i } v_ i ( x', \phi (x'))) \, dx' \\
& = & \int_{\partial\Omega} \left(uF - \sum_{i=1}^{n-1} v_i \frac{\partial u}{\partial \tau_i}
\right)\, d\sigma,
\end{eqnarray*}
where $\tau_i  = ( 1+ |\nabla \phi |^2) ^ { -1/2}  ( e_i + \phi _{ x_i } e_n) $ is a tangential vector, $v_i$ are the local Riesz
transforms, and $F = \sum _ { i =1 } ^ n \frac { \partial v_i } {
  \partial x_i} $ is the right-hand side of  (\ref{Divergence}).   Thus from Lemma
\ref{RieszEstimate} and Theorem \ref{Dahlberg}, it follows that
\begin{eqnarray*}
|\int_{\partial\Omega} u \frac{\partial v}{\partial \nu}\, d\sigma |
& \leq & \|u\|_{L^{1}(d\sigma)} \|F\|_{L^{\infty}(d\sigma)} +
\|v^*\|_{L^p (\mu \, d\sigma)} \|\nabla_t u\|_{L^{p'}(\mu^{-p'/p}\, d\sigma)} \\
& \leq &  C \|u \|_{W^{1,p'}(\mu^{-p'/p}\, d\sigma)} \| g\|_{L^p (\mu \, d\sigma)}.
\end{eqnarray*}
We give the details for the  estimate for the term $\|u\|_{ L^1 (
  d\sigma)} \|F\|_ {L^ \infty (
  d\sigma)}$. By a Poincar\'e inequality and H\"older's
inequality, we obtain
$$ \| u \| _ { L^ 1 ( d\sigma)} \leq C r \| \tangrad
u \|_ { L^ 1 ( d\sigma)} \leq C r \| \tangrad u \| _ { L^ {p'} ( \mu ^ {
    -p'/p}  d\sigma) }\mu ( \sball x r)^ { 1/p}.
$$
Recall that $F$  is the right-hand side of 
(\ref{Divergence}). Then it follows that
$$
\| F\|_ { L^ \infty ( d\sigma)} \leq C  r ^ { -1} 
\sup _K |v| \leq C  r ^ { -1}  \mu ( \sball x r ) ^ { -1/p }  \| \nontan v \| _ { L^
  p( \mu \,  d\sigma ) }.
$$
With these inequalities, the stated estimate follows.
\end{proof}

Before proving Theorem \ref{thmb} we require one more standard lemma.

\begin{lemma}\label{lemC}
If $\mu \in A_r(d\sigma)$, $1<r<\infty$, $\Delta u = 0$, and
$(\nabla u)^* \in L^2(d\sigma)$, then
\begin{equation*}
\int_{\partial\Omega} \left((\nabla u)^*\right)^r \, \mu \, d\sigma \leq
C\int_{\partial\Omega} |\nabla u|^r \, \mu \, d\sigma.
\end{equation*}
\end{lemma}

\begin{proof}
Let $\funsol $ be the fundamental solution for the Laplacian. We assume
that $(\nabla u)^* \in L^2(d\sigma)$. Under these conditions, it is
easy
to establish the representation formula
\begin{equation*}
\frac{\partial u}{\partial x_j}(x) = -\int_{\partial\Omega}
\frac{\partial \funsol}{\partial y_i} (x-\cdot) (\nu_i \frac{\partial u}{\partial y_j}
- \nu_j \frac{\partial u}{\partial y_i} ) + \frac{\partial \funsol}{\partial y_i} (x-\cdot)
\frac{\partial u}{\partial \nu} \, d\sigma.
\end{equation*}

The Lemma now follows from standard estimates on singular integral
operators on Lipschitz surfaces  \cite{CMM:1982} and weighted estimates
for Calder\'on-Zygmund operators \cite{CF:1974}.
\end{proof}

We are now prepared to prove Theorem \ref{thmb}.

\begin{proof}[Proof of Theorem \ref{thmb}]
Let $f\in W^{1,p'}(\mu^{-p'/p}\, d\sigma) \cap W^{1,2}(d\sigma)$
and let $u$ be the solution of the $L^2(d\sigma)$-regularity problem
with data $f$. 
From Lemma \ref{lma} and Lemma \ref{lemC}, we conclude that
\begin{equation*}
\int_{\partial\Omega} \left((\nabla u)^*\right)^{p'}
  \mu^{-p'/p}\,  d\sigma
 \leq   C \int_{\partial\Omega} (|\nabla_t u|^{p'} + |u|^{p'})
 \mu^{-p'/p}\,  d\sigma.
\end{equation*}
Now a limiting argument gives existence of solutions of the regularity problem
with boundary data $f\in W^{1,p'}(\mu^{-p'/p}\, d\sigma)$.
Since $L^{p'}(\mu^{-p'/p}\, d\sigma) \subset L^1(d\sigma)$, uniqueness follows
from Dahlberg and Kenig \cite{DK:1987}.
\end{proof}

\begin{theorem}\label{wm}
Let $t_0$ be as in the reverse H\"older inequality (\ref{rhomega}) and
$s_0 = t_0/(t_0-1)$.
Suppose that $\mu \in A_r(d\sigma)$, where $p>rs_0$. Set
$$
\alpha =\frac{p-1}{r-1}\qquad \mbox{and} \qquad
\theta = \frac{1}{\alpha'}=1-\frac{r-1}{p-1}= \frac {p-r}{p-1}.$$
Assume that
$1<p'<\theta(q_0/2)$. Consider the mixed problem with
Dirichlet data $f_D\in W^{1,p'}(\mu^{-p'/p}\, d\sigma)$ and
Neumann data $f_N \in L^{p'}(\mu^{-p'/p}\, d\sigma)$.
Then the following estimate for the
solution $u$ of the mixed problem holds
\begin{eqnarray}
\|(\nabla u)^*\|_{L^{p'}(\mu^{-p'/p}\, d\sigma)} & \leq &
C \left( \|f_N\|_{L^{p'}(\mu^{-p'/p}\, d\sigma)} +
\|f_D\|_{W^{1,p'}(\mu^{-p'/p}\, d\sigma)}\right).
\end{eqnarray}

Furthermore, there is only one solution satisfying $ ( \nabla u)^* \in
L^ { p'}( \mu ^ { -p'/p}\,d\sigma)$. 
\end{theorem}

\begin{proof}
In our application, we are given $\mu \in A_r(d\sigma)$ and
$p>r s_0$. We will apply (\ref{Shenconc})
with $\mu^{-p'/p} \in A_{p'}(d\sigma)$
and we observe that with $\alpha   = \frac{p-1}{r-1}$ we have
\begin{eqnarray*}
\int_{\Delta} (\mu^{-p'/p})^{\alpha}\, d\sigma & = &
\int_{\Delta} \mu^{-1/(r-1)} \, d\sigma \\
& \leq & \left( \int_{\Delta} \mu\, d\sigma \right)^{-1/(r-1)} \sigma(\Delta)^{r'}\\
& \leq & \left( \int_{\Delta} \mu^{-p'/p} \, d\sigma\right)^{(p-1)/(r-1)} \sigma(\Delta)^{r'}
\sigma(\Delta)^{-p'/(r-1)},
\end{eqnarray*}
for any surface ball $\Delta$. Therefore,
\begin{equation}\label{DIR-3}
\average_{\Delta} \mu^{(-p'\alpha)/p} \, d\sigma \leq
\left(\average_{\Delta} \mu^{-p'/p}\, d\sigma \right)^{\alpha},
\end{equation}
where we have used the $A_r$ condition for $ \mu$,
\begin{equation*}
\left(\int_{\Delta} \mu \, d\sigma \right)^{1/(r-1)}
\left(\int_{\Delta} \mu^{-1/(r-1)} \, d\sigma\right)
\leq \sigma(\Delta)^{r'},
\end{equation*}
and H\"older's inequality,
\begin{equation*}
\sigma(\Delta)^{1/(r-1)} \leq \left(\int_{\Delta} \mu\, d\sigma
\right)^{p  /(r-1)}
\left(\int_\Delta \mu^{-p'/p}\, d\sigma \right)^{(p-1)/(r-1)}.
\end{equation*}

From the inequality (\ref{DIR-3}) we have $\mu^{-p'/p} \in \rh _{\alpha}(d\sigma)$, and H\"older's
inequality implies
\begin{equation*}
\left(\frac{\mu(E)}{\mu(\Delta)}\right) \leq
\left(\frac{\sigma(E)}{\sigma(\Delta)}\right)^{\theta},
\quad \theta = 1-\frac{1}{\alpha} = \frac{p-r}{p-1}.
\end{equation*}
Thus, we have the conditions needed to obtain the conclusion
(\ref{Shenconc}).

To summarize, if $1<p'<\theta q_0/2$, with $q_0$ as in Lemma
\ref{RHEstimate}, and if $\mu \in A_r(d\sigma)$,
$p>rs_0$, with $s_0$ as in Theorem \ref{Dahlberg},
then we have the solution of the mixed problem
satisfies
\begin{equation*}
\|(\nabla u)^* \|_{L^{p}(\mu^{-p'/p}\, d\sigma)} \leq
C\left( \|f_N \|_{L^{p'}(\mu^{-p'/p} \,  d\sigma)} +
\|f_D \|_{W^{1,p'}(\mu^{-p'/p}\, d\sigma)}\right).
\end{equation*}

Since $ L^ { p'} ( \mu ^ { -p'/p} ) \subset L^ 1 ( d\sigma)$,
uniqueness for solutions of the mixed problem with $ ( \nabla u )^ *
\in L^ { p'}( \mu ^ { -p'/p}\, d\sigma)$ follows from the 
problem follows from part a) of Theorem 1.1. 

\end{proof}

\note{ Index of notation.

\begin{tabular}{rr}
\sc Symbol &  Meaning \rm \\
$D$ & region where we specify Dirichlet data\\
$N$ & region where we specify Neumann data \\
$\Lambda$ & boundary of $D$ relative to $\partial\Omega$\\
$\delta(y)$ & distance from $y$ to $\Lambda$\\
$dist(y,\partial\Omega)$ & distance from $y$ to $\partial\Omega$\\

\end{tabular}
}

\begin{thebibliography}{10}

\bibitem{MR2735986}
C.~B{\u{a}}cu{\c{t}}{\u{a}}, A.L. Mazzucato, V.~Nistor, and L.~Zikatanov.
\newblock Interface and mixed boundary value problems on {$n$}-dimensional
  polyhedral domains.
\newblock {\em Doc. Math.}, 15:687--745, 2010.

\bibitem{RB:1994b}
R.M. Brown.
\newblock The mixed problem for {L}aplace's equation in a class of {L}ipschitz
  domains.
\newblock {\em Comm. Partial Diff. Eqns.}, 19:1217--1233, 1994.

\bibitem{MR1486629}
L.A. Caffarelli and I.~Peral.
\newblock On {$W\sp {1,p}$} estimates for elliptic equations in divergence
  form.
\newblock {\em Comm. Pure Appl. Math.}, 51(1):1--21, 1998.

\bibitem{MR1190215}
S.~Chanillo and Y.Y. Li.
\newblock Continuity of solutions of uniformly elliptic equations in {${\bf
  R}^2$}.
\newblock {\em Manuscripta Math.}, 77(4):415--433, 1992.

\bibitem{CF:1974}
R.R. Coifman and C.~Fefferman.
\newblock Weighted norm inequalities for maximal functions and singular
  integrals.
\newblock {\em Studia Math.}, 51:241--250, 1974.

\bibitem{CMM:1982}
R.R. Coifman, A.~McIntosh, and Y.~Meyer.
\newblock L'int\'egrale de {C}auchy d\'efinit un op\'erateur born\'e sur
  ${L^2}$ pour les courbes lipschitziennes.
\newblock {\em Ann. of Math.}, 116:361--387, 1982.

\bibitem{BD:1977}
B.E.J. Dahlberg.
\newblock Estimates of harmonic measure.
\newblock {\em Arch. Rational Mech. Anal.}, 65(3):275--288, 1977.

\bibitem{BD:1980}
B.E.J. Dahlberg.
\newblock Weighted norm inequalities for the {L}usin area integral and
  nontangential maximal functions for functions harmonic in a {L}ipschitz
  domain.
\newblock {\em Studia Math.}, 67:297--314, 1980.

\bibitem{DK:1987}
B.E.J. Dahlberg and C.E. Kenig.
\newblock Hardy spaces and the {N}eumann problem in ${L^p}$ for {L}aplace's
  equation in {L}ipschitz domains.
\newblock {\em Ann. of Math.}, 125:437--466, 1987.

\bibitem{EG:1957}
E.~De~Giorgi.
\newblock Sulla differenziabilit\`a e l'analiticit\`a delle estremali degli
  integrali multipli regolari.
\newblock {\em Mem. Accad. Sci. Torino. Cl. Sci. Fis. Mat. Nat. (3)}, 3:25--43,
  1957.

\bibitem{MR990595}
K.~Gr{\"o}ger.
\newblock A {$W^{1,p}$}-estimate for solutions to mixed boundary value problems
  for second order elliptic differential equations.
\newblock {\em Math. Ann.}, 283(4):679--687, 1989.

\bibitem{MR657523}
M.~Gr{\"u}ter and K.O. Widman.
\newblock The {G}reen function for uniformly elliptic equations.
\newblock {\em Manuscripta Math.}, 37(3):303--342, 1982.

\bibitem{MR0274787}
R.A. Hunt and R.L. Wheeden.
\newblock Positive harmonic functions on {L}ipschitz domains.
\newblock {\em Trans. Amer. Math. Soc.}, 147:507--527, 1970.

\bibitem{JK:1982a}
D.S. Jerison and C.E. Kenig.
\newblock Boundary value problems in {L}ipschitz domains.
\newblock In Walter Littman, editor, {\em Studies in partial differential
  equations}, volume~23 of {\em MAA Studies in Mathematics}, pages 1--68.
  {Math. Assoc. Amer.}, Washington, D.C., 1982.

\bibitem{JK:1982c}
D.S. Jerison and C.E. Kenig.
\newblock The {N}eumann problem on {L}ipschitz domains.
\newblock {\em Bull. Amer. Math. Soc.}, 4:203--207, 1982.

\bibitem{CK:1994}
C.E. Kenig.
\newblock {\em Harmonic analysis techniques for second order elliptic boundary
  value problems}.
\newblock Published for the Conference Board of the Mathematical Sciences,
  Washington, DC, 1994.

\bibitem{MR87f:35065}
C.E. Kenig and W.M. Ni.
\newblock On the elliptic equation ${L}u-k+{K}\,{\rm exp}[2u]=0$.
\newblock {\em Ann. Scuola Norm. Sup. Pisa Cl. Sci. (4)}, 12(2):191--224, 1985.

\bibitem{KS:2010}
J.~Kilty and Z.~Shen.
\newblock The {$ L^p$} regularity problem on {L}ipschitz domains.
\newblock {\em Trans. Amer. Math. Soc.}, 2010.

\bibitem{MR0244627}
O.A. Ladyzhenskaya and N.N. Ural{\cprime}tseva.
\newblock {\em Linear and quasilinear elliptic equations}.
\newblock Translated from the Russian by Scripta Technica, Inc. Translation
  editor: Leon Ehrenpreis. Academic Press, New York, 1968.

\bibitem{LCB:2008}
L.~Lanzani, L.~Capogna, and R.M. Brown.
\newblock The mixed problem in {$L\sp p$} for some two-dimensional {L}ipschitz
  domains.
\newblock {\em Math. Ann.}, 342(1):91--124, 2008.

\bibitem{MR2442898}
J.~Lehrb{\"a}ck.
\newblock Weighted {H}ardy inequalities and the size of the boundary.
\newblock {\em Manuscripta Math.}, 127(2):249--273, 2008.

\bibitem{NM:1963}
N.G. Meyers.
\newblock An {$L\sp{p}$}-estimate for the gradient of solutions of second order
  elliptic divergence equations.
\newblock {\em Ann. Scuola Norm. Sup. Pisa (3)}, 17:189--206, 1963.

\bibitem{MR1883390}
D.~Mitrea.
\newblock Layer potentials and {H}odge decompositions in two dimensional
  {L}ipschitz domains.
\newblock {\em Math. Ann.}, 322(1):75--101, 2002.

\bibitem{MM:2007}
I.~Mitrea and M.~Mitrea.
\newblock The {P}oisson problem with mixed boundary conditions in {S}obolev and
  {B}esov spaces in non-smooth domains.
\newblock {\em Trans. Amer. Math. Soc.}, 359(9):4143--4182 (electronic), 2007.

\bibitem{JM:1961}
J.~Moser.
\newblock On {H}arnack's theorem for elliptic differential operators.
\newblock {\em Comm. Pure Appl. Math.}, 14:577--591, 1961.

\bibitem{OB:2009}
K.A. Ott and R.M. Brown.
\newblock The mixed problem for the {L}aplacian in {L}ipschitz domains.
\newblock arXiv:0909.0061 [math.AP], 2009.

\bibitem{ZS:2005}
Z.~Shen.
\newblock Weighted estimates in {$L^2$} for {L}aplace's equation on {L}ipschitz
  domains.
\newblock {\em Trans. Amer. Math. Soc.}, 357:2843--2870, 2005.

\bibitem{ZS:2007}
Z.~Shen.
\newblock The {$L\sp p$} boundary value problems on {L}ipschitz domains.
\newblock {\em Adv. Math.}, 216(1):212--254, 2007.

\bibitem{GS:1960}
G.~Stampacchia.
\newblock Problemi al contorno ellitici, con dati discontinui, dotati di
  soluzionie h\"olderiane.
\newblock {\em Ann. Mat. Pura Appl. (4)}, 51:1--37, 1960.

\bibitem{ES:1970}
E.M. Stein.
\newblock {\em Singular integrals and differentiability properties of
  functions}.
\newblock Princeton University Press, Princeton NJ, 1970.

\bibitem{SB:2001}
J.D. Sykes and R.M. Brown.
\newblock The mixed boundary problem in {$L\sp p$} and {H}ardy spaces for
  {L}aplace's equation on a {L}ipschitz domain.
\newblock In {\em Harmonic analysis and boundary value problems (Fayetteville,
  AR, 2000)}, volume 277 of {\em Contemp. Math.}, pages 1--18. Amer. Math.
  Soc., Providence, RI, 2001.

\bibitem{VV:2008}
M.~Venouziu and G.C. Verchota.
\newblock The mixed problem for functions in polyhedra of {${\bf R}^3$}.
\newblock Preprint 2008.

\bibitem{GV:1982}
G.C. Verchota.
\newblock {\em Layer potentials and boundary value problems for {L}aplace's
  equation on {L}ipschitz domains}.
\newblock PhD thesis, University of Minnesota, 1982.

\bibitem{GV:1984}
G.C. Verchota.
\newblock Layer potentials and regularity for the {D}irichlet problem for
  {L}aplace's equation on {L}ipschitz domains.
\newblock {\em J. Funct. Anal.}, 59:572--611, 1984.

\end{thebibliography}
\def\cprime{$'$} \def\cprime{$'$} \def\cprime{$'$} \def\cprime{$'$}


\end{document}